\documentclass[11pt,a4paper]{article}

\usepackage{graphicx} \usepackage{tikz} \usepackage{tikz-cd} 

\usepackage{amsmath,amssymb,amsthm,thmtools} \usepackage{bm} \usepackage{mathrsfs} \usepackage{newtxtext,newtxmath} \usepackage{mathtools} \usepackage{halloweenmath} \usepackage{physics} \usepackage{lipsum} \usepackage{etoolbox} 

\usepackage{mdframed} \usepackage{tcolorbox} 

\usepackage{enumitem} 
\usepackage{xcolor} \usepackage{hyperref} \usepackage{cleveref} \usepackage{fancyhdr} \usepackage{lastpage} \usepackage{setspace} 

\usepackage[utf8]{inputenc}
\usepackage[autolanguage]{numprint}

\DeclareMathOperator{\GL}{\mathbf{GL}}

\DeclareMathOperator{\Spec}{Spec}

\DeclareMathOperator{\Gal}{Gal}

\DeclareMathOperator{\N}{N}

\DeclareMathOperator{\Ker}{\mathrm{Ker}\,}
\DeclareMathOperator{\Coker}{\mathrm{Coker}\,}
\DeclareMathOperator{\Inn}{Inn}
\DeclareMathOperator{\Isom}{Isom}
\DeclareMathOperator{\Out}{Out}
\DeclareMathOperator{\Aut}{Aut}
\DeclareMathOperator{\Art}{Art}
\DeclareMathOperator{\Ver}{Ver}
\DeclareMathOperator{\ord}{ord}
\DeclareMathOperator{\id}{id}
\DeclareMathOperator{\Frob}{Frob}

\newcommand{\ab}{\mathrm{ab}}
\newcommand{\nr}{\mathrm{un}}
\newcommand{\un}{\mathrm{un}}
\newcommand{\cyc}{\mathrm{cycl}}
\newcommand{\cycl}{\mathrm{cycl}}
\newcommand{\tame}{\mathrm{tame}}
\newcommand{\alg}{\mathrm{alg}}
\newcommand{\sep}{\mathrm{sep}}
\newcommand{\tor}{\mathrm{tor}}

\newcommand{\MLF}{\mathrm{MLF}}

\newcommand{\HT}{\mathrm{HT}}
\newcommand{\Primes}{\mathfrak{Primes}}

\newcommand{\IsomFilt}{\mathrm{Isom}_{\mathit{filt.}}}
\newcommand{\OutFilt}{\mathrm{Out}_{\mathit{filt.}}}

\newcommand{\CC}{\mathbf{C}}

\newcommand{\FF}{\mathbf{F}}

\newcommand{\QQ}{\mathbf{Q}}

\newcommand{\ZZ}{\mathbf{Z}}

 \newcommand{\termTheorem}{Theorem}
\newcommand{\termTheorems}{Theorems}

\newcommand{\termProposition}{Proposition}
\newcommand{\termPropositions}{Propositions}
\newcommand{\termLemma}{Lemma}
\newcommand{\termLemmas}{Lemmas}
\newcommand{\termConjecture}{Conjecture}
\newcommand{\termConjectures}{Conjectures}
\newcommand{\termCorollary}{Corollary}
\newcommand{\termCorollaries}{Corollaries}
\newcommand{\termDefinition}{Definition}
\newcommand{\termDefinitions}{Definitions}
\newcommand{\termProblem}{Problem}
\newcommand{\termProblems}{Problems}
\newcommand{\termSolution}{Solution}

\newcommand{\termExample}{Example}
\newcommand{\termExamples}{Examples}
\newcommand{\termRemark}{Remark}

\newcommand{\termNote}{Note}

\newcommand{\termEquation}{Equation}
\newcommand{\termEquations}{Equations}
\newcommand{\termFigure}{Figure}
\newcommand{\termFigures}{Figures}
\newcommand{\termTable}{Table}
\newcommand{\termTables}{Tables}

 \newtheoremstyle{latinplain}{16pt}{16pt}{\itshape}{}{\bfseries}{.}{.5em}{}

\makeatletter
\def\cleartheorem#1{\expandafter\let\csname#1\endcsname\relax
    \expandafter\let\csname c@#1\endcsname\relax
}
\makeatother

\theoremstyle{latinplain}

\newtheorem{theorem}{\termTheorem}[section]
\newtheorem*{theorem*}{\termTheorem}
\newtheorem{proposition}[theorem]{\termProposition}
\newtheorem*{proposition*}{\termProposition}
\newtheorem{lemma}[theorem]{\termLemma}
\newtheorem*{lemma*}{\termLemma}

\newtheorem*{conjecture*}{\termConjecture}
\newtheorem{corollary}[theorem]{\termCorollary}
\newtheorem*{corollary*}{\termCorollary}

\newtheorem*{problem*}{\termProblem}

\theoremstyle{definition}

\newtheorem{definition}[theorem]{\termDefinition}
\newtheorem*{definition*}{\termDefinition}
\newtheorem{example}[theorem]{\termExample}
\newtheorem*{example*}{\termExample}

\theoremstyle{remark}

\newtheorem*{remark}{\termRemark}

\newcommand{\envendmark}{\hfill$\diamond$}

\AtEndEnvironment{theorem}{\envendmark}
\AtEndEnvironment{theorem*}{\envendmark}
\AtEndEnvironment{proposition}{\envendmark}
\AtEndEnvironment{proposition*}{\envendmark}
\AtEndEnvironment{lemma}{\envendmark}
\AtEndEnvironment{lemma*}{\envendmark}
\AtEndEnvironment{conjecture}{\envendmark}
\AtEndEnvironment{conjecture*}{\envendmark}
\AtEndEnvironment{corollary}{\envendmark}
\AtEndEnvironment{corollary*}{\envendmark}
\AtEndEnvironment{problem}{\envendmark}
\AtEndEnvironment{problem*}{\envendmark}

\AtEndEnvironment{definition}{\envendmark}
\AtEndEnvironment{definition*}{\envendmark}
\AtEndEnvironment{example}{\envendmark}
\AtEndEnvironment{example*}{\envendmark}

\AtEndEnvironment{numbered}{\envendmark}

\AtEndEnvironment{note}{\envendmark}
\AtEndEnvironment{remark}{\envendmark}
\AtEndEnvironment{solution}{\envendmark}

 \crefname{section}{\S\!}{\S\S\!}
\crefname{subsection}{\S\!}{\S\S\!}

\crefname{theorem}{\MakeLowercase{\termTheorem}}{\MakeLowercase{\termTheorems}}
\crefname{proposition}{\MakeLowercase{\termProposition}}{\MakeLowercase{\termPropositions}}
\crefname{lemma}{\MakeLowercase{\termLemma}}{\MakeLowercase{\termLemmas}}
\crefname{conjecture}{\MakeLowercase{\termConjecture}}{\MakeLowercase{\termConjectures}}
\crefname{corollary}{\MakeLowercase{\termCorollary}}{\MakeLowercase{\termCorollaries}}
\crefname{problem}{\MakeLowercase{\termProblem}}{\MakeLowercase{\termProblems}}

\crefname{definition}{\MakeLowercase{\termDefinition}}{\MakeLowercase{\termDefinitions}}
\crefname{example}{\MakeLowercase{\termExample}}{\MakeLowercase{\termExamples}}

\crefname{equation}{\MakeLowercase{\termEquation}}{\MakeLowercase{\termEquations}}
\crefname{figure}{\MakeLowercase{\termFigure}}{\MakeLowercase{\termFigures}}
\crefname{table}{\MakeLowercase{\termTable}}{\MakeLowercase{\termTables}}

\Crefname{section}{\S\!}{\S\S\!}

\Crefname{theorem}{\termTheorem}{\termTheorems}
\Crefname{proposition}{\termProposition}{\termPropositions}
\Crefname{lemma}{\termLemma}{\termLemmas}
\Crefname{conjecture}{\termConjecture}{\termConjectures}
\Crefname{corollary}{\termCorollary}{\termCorollaries}
\Crefname{problem}{\termProblem}{\termProblems}

\Crefname{definition}{\termDefinition}{\termDefinitions}
\Crefname{example}{\termExample}{\termExamples}

\Crefname{equation}{\termEquation}{\termEquations}
\Crefname{figure}{\termFigure}{\termFigures}
\Crefname{table}{\termTable}{\termTables}
 
\usepackage[
    top=25truemm,
    bottom=25truemm,
    left=25truemm,
    right=25truemm
]{geometry}

\title{The $m$-step solvable anabelian geometry of \\ mixed-characteristic local fields\thanks{Mathematics Subject Classification (MSC2020): 11S20, 11S31, 11S15, 11F80}}
\author{\textsc{Hyeon} Seung-Hyeon\thanks{Department of Mathematics,
Institute of Science Tokyo, 2-12-1 Ookayama, Meguro, Tokyo 152-8550, Japan
}\:\:\thanks{Email: \texttt{hyun.s.887e@m.isct.ac.jp}}}

\usepackage[english]{babel}

\usepackage[
    backend=biber,
    style=numeric-verb,
    giveninits=true,
]{biblatex}
\usepackage{xurl} \hypersetup{breaklinks=true}

\usepackage{quiver}

\begin{document}

\maketitle

\begin{abstract}
    \noindent Let $K$ be a mixed-characteristic local field.
    For an integer $m \geq 0$, we denote by $K^m / K$ the maximal $m$-step solvable extension of $K$, and by $G_K^m$ the maximal $m$-step solvable quotient of the absolute Galois group $G_K$ of $K$.
    We regard $G_K$ and its quotients as filtered profinite groups via the respective upper-numbering ramification filtrations.
    It is known from the previous result due to Mochizuki that the isomorphism class of $K$ is determined by the isomorphism class of the filtered profinite group $G_K$.
    In this paper, we prove that the isomorphism class of $K$ is determined by the isomorphism class of the maximal $2$-step solvable quotient $G_K^2$ as a filtered profinite group, and furthermore, that $K^m / K$ is determined functorially by the filtered profinite group $G_K^{m + 2}$ (resp. $G_K^{m + 3}$) for $m \geq 2$ (resp. $m = 0, 1$).
\end{abstract}

\tableofcontents

\section{Introduction}
\label{section:2025.10.15.11.02.40}

\emph{Anabelian geometry} is a branch of arithmetic geometry that studies how the arithmetic of a geometric object $X$ is encoded in its (étale) fundamental group.
The central philosophy, originally proposed by Grothendieck \cite{Grothendieck1997}, is that certain geometric objects of an \emph{anabelian} nature should admit characterizations in the language of fundamental groups.
This translates to the principle that a field $k$ of a certain type should be determined by its absolute Galois group $G_k$ (as a profinite group), upon taking $X = \Spec k$.

The Neukirch-Uchida theorem---which predated even the coining of the term ``anabelian''---is one of the earliest validations of this philosophy.
The theorem states that \emph{number fields} can be determined by their absolute Galois groups \cite[Corollary 2]{Uchida1976}, i.e., it holds that, for two number fields $F_\circ$ and $F_\bullet$,
\begin{equation*}
F_\circ
        \cong_{\textrm{fields}}
            F_\bullet \quad
    \Leftrightarrow \quad
G_{F_\circ}
        \cong_{\textrm{prof. gr.}}
            G_{F_\bullet}.
\end{equation*}
This remarkable result naturally led to investigations into its local analogues.
Most notably, Mochizuki \cite{Mochizuki1997} established the following result concerning \emph{mixed-characteristic local fields}: if we regard the Galois groups as \emph{filtered} profinite groups via the upper-numbering ramification groups (cf. \cite[Chap. IV, \S3]{Serre1979}), then, for two mixed-characteristic local fields $K_\circ$ and $K_\bullet$,
\begin{equation*}
K_\circ
        \cong_{\textrm{fields}}
            K_\bullet \quad
    \Leftrightarrow \quad
G_{K_\circ}
        \cong_{\textit{filt.}\textrm{ prof. gr.}}
            G_{K_\bullet}.
\end{equation*}
The ``filtration-preserving condition'' is indispensable here: Yamagata \cite{Yamagata1976} and Jarden-Ritter \cite{JardenRitter1979}, for example, constructed pairs of non-isomorphic mixed-characteristic local fields whose absolute Galois groups are isomorphic as profinite groups.

Meanwhile, there has been an extensive study of whether Neukirch-Uchida-type results hold for various quotients of absolute Galois groups.
Saïdi-Tamagawa \cite{SaidiTamagawa2022} showed that number fields can also be determined by the \emph{maximal $m$-step solvable quotients} (cf. p.\pageref{2025.10.15.11.33.37}) of their absolute Galois groups if $m \geq 3$.
That is, for number fields $F_\circ$ and $F_\bullet$,
\begin{equation*}
F_\circ
        \cong_{\textrm{fields}}
            F_\bullet \quad
    \Leftrightarrow \quad
G_{F_\circ}^3
        \cong_{\textrm{prof. gr.}}
            G_{F_\bullet}^3.
\end{equation*}
Remarkably, these quotients retain enough arithmetic information to determine the field structures, despite carrying less information than the full absolute Galois groups.
However, Koymans-Pagano \cite{KoymansPagano2024} demonstrated that the analogue for \emph{maximal $2$-step nilpotent (i.e., abelian-by-central) quotients} fails: one may choose $F_\circ$ and $F_\bullet$ from
\begin{equation*}
    \QQ(\sqrt{-11}),\quad
    \QQ(\sqrt{-19}),\quad
    \QQ(\sqrt{-43}),\quad
    \QQ(\sqrt{-67}),\quad
    \QQ(\sqrt{-163})
\end{equation*}
so that the maximal $2$-step nilpotent quotients of $G_{F_\circ}$ and $G_{F_\bullet}$ are isomorphic.

\subsection*{Main results}

In this paper, we prove the following: mixed-characteristic local fields can be determined by the \emph{maximal $2$-step solvable (i.e., metabelian) quotients} of their absolute Galois groups as \emph{filtered} profinite groups.
That is, for mixed-characteristic local fields $K_\circ$ and $K_\bullet$,
\begin{equation*}
K_\circ
        \cong_{\textrm{fields}}
            K_\bullet \quad
    \Leftrightarrow\quad
G_{K_\circ}^2
        \cong_{\textit{filt.}\textrm{ prof. gr.}}
            G_{K_\bullet}^2.
\end{equation*}

We begin by recalling the precise statement of Mochizuki's result.
Let $K_\circ$ (resp. $K_\bullet$) be a mixed-characteristic local field of residue characteristic $p_{K_\circ}$ (resp. $p_{K_\bullet}$).
For each $\square \in \left\{ \circ, \bullet \right\}$, we fix an algebraic closure $K_\square^\alg$ of $K_\square$, and we regard the absolute Galois group $G_{K_\square} = \Gal ({K_\square^\alg} / K_\square)$ of $K_\square$ and its quotients as filtered profinite groups via the upper-numbering ramification groups.
Suppose we are given a field isomorphism $f$ from $K_\circ$ into $K_\bullet$.
Then we have $p_{K_\circ} = p_{K_\bullet} \, (\eqqcolon p)$ (see, e.g., \cref{section:2025.10.15.11.33.53}) and $f$ is, in particular, a $\QQ_p$-algebra isomorphism.
We can choose an isomorphism $\theta \colon {K_\circ^\alg} \to {K_\bullet^\alg}$ that extends $f$; it defines a profinite group isomorphism
\begin{equation}
    \label{equation:2025.10.16.23.28.58}
\sigma
    \mapsto
\theta
        \circ
            \sigma
        \circ
            \theta^{-1}
\end{equation}
from $G_{K_\circ}$ to $G_{K_\bullet}$.
We denote by
\begin{equation*}
\eta (f)
    \in
        \Out (G_{K_\circ}, G_{K_\bullet})
        \coloneqq
\Inn (G_{K_\bullet})
            \backslash
                \Isom (G_{K_\circ}, G_{K_\bullet})
\end{equation*}
the equivalence class of the above isomorphism modulo inner automorphisms of $G_{K_\bullet}$ (i.e., the \emph{outer isomorphism} defined by the above isomorphism).
Here, $\Isom (G_{K_\circ}, G_{K_\bullet})$ (resp. $\Out (G_{K_\circ}, G_{K_\bullet})$) denotes the set of isomorphisms (resp. outer isomorphisms) $G_{K_\circ} \to G_{K_\bullet}$ of profinite groups.
We see that $\eta (f)$ does not depend on the choice of the extension $\theta$; therefore, we obtain a natural map
\begin{equation*}
    \eta
    \colon
        \Isom (K_\circ, K_\bullet)
    \to
        \Out (G_{K_\circ}, G_{K_\bullet}).
\end{equation*}
The injectivity of $\eta$ follows directly from the \emph{slimness} of $G_{K_\circ}$ (cf. \cite[Proposition 2.1]{Hoshi2022}); however, $\eta$ is \emph{not} surjective in general (cf. \cite[Chap. XII, \S2]{NeukirchSchmidtWingberg2008}, see also \cite{HoshiNishio2022}).

Moreover, since any field isomorphism between $K_\circ$ and $K_\bullet$ preserves the $p$-adic valuation, one can verify that the isomorphism \eqref{equation:2025.10.16.23.28.58} is filtration-preserving.
Therefore, $\eta$ factors through
\begin{equation*}
\OutFilt (G_{K_\circ}, G_{K_\bullet})
    \coloneqq
\Inn (G_{K_\bullet})
        \backslash
            \IsomFilt (G_{K_\circ}, G_{K_\bullet})
\end{equation*}

\begin{theorem}[Mochizuki \cite{Mochizuki1997}]
    \label{theorem:2025.10.15.11.05.33}
    The image of $\eta$ coincides with $\OutFilt (G_{K_\circ}, G_{K_\bullet})$.
    Equivalently, for an isomorphism
    \begin{equation*}
        \alpha
        \colon
            G_{K_\circ}
        \xrightarrow{\cong}
            G_{K_\bullet}
    \end{equation*}
    of \emph{filtered} profinite groups, there exists a unique isomorphism
    $\theta \colon {K_\circ^\alg} \to {K_\bullet^\alg}$ such that
    \begin{equation*}
\alpha (\sigma)
        =
\theta
            \circ
                \sigma
            \circ
                \theta^{-1}
    \end{equation*}
    for every $\sigma \in G_{K_\circ}$. In particular, we have an isomorphism $\theta \vert_{K_\circ} \colon K_{\circ} \to K_{\bullet}$.
\end{theorem}

This result is one form of the \emph{Grothendieck Conjecture} for mixed-characteristic local fields: the \emph{filtration-preserving} isomorphisms between $G_{K_\circ}$ and $G_{K_\bullet}$ are precisely the \emph{geometric} isomorphisms, i.e., those induced by field isomorphisms between $K_\circ$ and $K_\bullet$.

\Cref{theorem:2025.10.15.11.05.33} was later extended by Abrashkin, albeit with a different method, to the case of \emph{equal-characteristic local fields} \cite{Abrashkin2000,Abrashkin2010} and \emph{higher-dimensional local fields of odd residue characteristic} \cite[Theorems 5 and 6]{Abrashkin2007}.

We turn to the result of Saïdi-Tamagawa.
Let $F_\circ$ and $F_\bullet$ be number fields.
For $\square \in \{ \circ, \bullet \}$, we denote by $F_\square^m$ the maximal $m$-step solvable extension of $F_\square$.

\begin{theorem}[Saïdi-Tamagawa {\cite[Theorem 1]{SaidiTamagawa2022}}]
    \label{theorem:2025.10.17.01.15.41}
    Assume that there exists an isomorphism
    \begin{equation*}
        A_3
        \colon
            G_{F_\circ}^3
        \overset{\cong}{\to}
            G_{F_\bullet}^3
    \end{equation*}
    of profinite groups.
    Then there exists an isomorphism $h \colon F_\circ \to F_\bullet$.
\end{theorem}

\begin{theorem}[Saïdi-Tamagawa {\cite[Theorem 2]{SaidiTamagawa2022}}]
    \label{theorem:2025.10.17.01.15.55}
    Let $m$ be an integer $\geq 0$.
    For an isomorphism
    \begin{equation*}
        A_{m + 4}
        \colon
            G_{F_\circ}^{m + 4}
        \overset{\cong}{\to}
            G_{F_\bullet}^{m + 4}
    \end{equation*}
    of profinite groups, there exists an isomorphism $\varTheta_{m + 1} \colon F_\circ^{m + 1} \to F_\bullet^{m + 1}$ such that
    \begin{equation*}
A_{m + 1} (\sigma)
        =
\varTheta_{m + 1}
            \circ
                \sigma
            \circ
                \varTheta_{m + 1}^{-1}
    \end{equation*}
    for every $\sigma \in G_{F_\circ}^{m + 1}$, where $A_{m + 1} \colon G_{F_\circ}^{m + 1} \to G_{F_\bullet}^{m + 1}$ is the isomorphism induced by $A_{m + 4}$.
    Moreover,
    \begin{itemize}
        \item if $m \geq 1$, the isomorphism $\varTheta_{m + 1}$ is uniquely determined by $A_{m + 4}$;
        \item if $m = 0$, the isomorphism $\varTheta_{m + 1} \vert_{F_\circ}\colon F_\circ \to F_\bullet$ is uniquely determined by $A_{m + 4}$.
    \end{itemize}
\end{theorem}

The statement of \cref{theorem:2025.10.17.01.15.41} lacks functoriality, meaning that there is no clear description of how $A_3$ and $h$ are related to each other, which makes it a \emph{weak bi-anabelian} result.
In contrast, the isomorphism class of a given number field $F$ is \emph{functorially} determined from the isomorphism class of $G_{F}^{4}$ in \cref{theorem:2025.10.17.01.15.55}; hence one might claim that \cref{theorem:2025.10.17.01.15.55} is a \emph{strong bi-anabelian} result.

We now proceed to the main results of this paper, which are the respective local counterparts of \cref{theorem:2025.10.17.01.15.41,theorem:2025.10.17.01.15.55}.

\begin{theorem}[\Cref{theorem:2025.10.22.02.13.03}]
    \label{theorem:2025.10.15.11.32.12}
    Assume that there exists an isomorphism
    \begin{equation*}
        \alpha_2
        \colon
            G_{K_\circ}^2
        \overset{\cong}{\to}
            G_{K_\bullet}^2
    \end{equation*}
    of \emph{filtered} profinite groups.
    Then there exists an isomorphism $f \colon K_\circ \to K_\bullet$.
\end{theorem}

\begin{theorem}[\Cref{theorem:2025.10.22.02.13.10}]
    \label{theorem:2025.10.15.11.32.25}
    Let $m$ be an integer $\geq 0$.
    For an isomorphism
    \begin{equation*}
        \alpha_{m + 3}
        \colon
            G_{K_\circ}^{m + 3}
        \xrightarrow{\cong}
            G_{K_\bullet}^{m + 3}
    \end{equation*}
    of \emph{filtered} profinite groups, there exists an isomorphism
    $\theta_{m + 1} \colon K_{\circ}^{m + 1} \to K_{\bullet}^{m + 1}$ such that
    \begin{equation*}
\alpha_{m + 1} (\sigma)
        =
\theta_{m + 1}
            \circ
                \sigma
            \circ
                \theta_{m + 1}^{-1}
    \end{equation*}
    for every $\sigma \in G_{K_\circ}^{m + 1}$, where $\alpha_{m + 1} \colon G_{K_\circ}^{m + 1} \to G_{K_\bullet}^{m + 1}$ is the isomorphism induced by $\alpha_{m + 3}$.
    Moreover,
    \begin{enumerate}[label=(\roman*)]
        \item if $m \geq 1$, the isomorphism $\theta_{m + 1}$ is uniquely determined by $\alpha_{m + 3}$;
        \item if $m = 0$, the isomorphism $\theta_{m + 1} \vert_{K_\circ} \colon K_\circ \to K_\bullet$ is uniquely determined by $\alpha_{m + 3}$.
    \end{enumerate}
\end{theorem}

As with Saïdi-Tamagawa's results for number fields (\cref{theorem:2025.10.17.01.15.41,theorem:2025.10.17.01.15.55}), we distinguish between weak and strong bi-anabelian results.
The statement of \cref{theorem:2025.10.15.11.32.12} lacks functoriality---there is no clear description of how $\alpha_2$ and $f$ are related---making it a \emph{weak bi-anabelian} result analogous to \cref{theorem:2025.10.17.01.15.41}.
(It also remains open whether there exists a \emph{filtration-preserving} isomorphism $G_{K_\circ}^2 \to G_{K_\bullet}^2$ that is \emph{not geometric}.)
In contrast, \cref{theorem:2025.10.15.11.32.25} functorially determines the isomorphism class of a given mixed-characteristic local field $K$ from $G_K^3$, providing a \emph{strong bi-anabelian} result that parallels \cref{theorem:2025.10.17.01.15.55}.
(To be precise, \cref{theorem:2025.10.15.11.32.25} asserts that any isomorphism $G_{K_\circ}^{m + 1} \to G_{K_\bullet}^{m + 1}$ \emph{induced by a filtration-preserving isomorphism $G_{K_\circ}^{m + 3} \to G_{K_\bullet}^{m + 3}$} must be \emph{geometric}, leaving open the question of whether $m + 3$ can be replaced by a smaller degree.)

In light of the developments so far, \cref{theorem:2025.10.15.11.32.12,theorem:2025.10.15.11.32.25} can be viewed as
\begin{itemize}
    \item local analogues of \cref{theorem:2025.10.17.01.15.41,theorem:2025.10.17.01.15.55}, respectively;
    \item $m$-step solvable variants of \cref{theorem:2025.10.15.11.05.33}---weak and strong bi-anabelian, respectively.
\end{itemize}

\begin{equation*}
    \begin{tikzpicture}
    \matrix (m) [
        matrix of nodes,
        nodes={draw},
        column sep=9mm,
        row sep=15mm,
    ] {
        & Neukirch-Uchida & \\
        Mochizuki \cite{Mochizuki1997} & & Saïdi-Tamagawa \cite{SaidiTamagawa2022} \\
        & \Cref{theorem:2025.10.15.11.32.12,theorem:2025.10.15.11.32.25} & \\
    };
    \begin{scope}[
        font=\footnotesize,
        inner sep=.25em,
        every node/.style={fill=white},
        shorten >=3mm,
        shorten <=3mm,
        ]
        ;
        \draw[-to](m-1-2) -- node {local analogue} (m-2-1);
        \draw[-to](m-1-2) -- node {$m$-step solvable version ($m\geq 3$)} (m-2-3);
        \draw[-to](m-2-1) -- node {$m$-step solvable version ($m\geq 2$)} (m-3-2);
        \draw[-to](m-2-3) -- node {local analogue} (m-3-2);
    \end{scope}
    \end{tikzpicture}
\end{equation*}

We prove \cref{theorem:2025.10.15.11.32.12,theorem:2025.10.15.11.32.25} in \cref{section:2025.10.15.11.36.45}.
The proof of \cref{theorem:2025.10.15.11.32.12} can be thought of as an application of $p$-adic Hodge theory.
In fact, we implement only a few adjustments to the method developed by Mochizuki.
For instance, in \cref{section:2025.10.15.11.35.27}, we show that, for a mixed-characteristic local field $K$, the \emph{Hodge-Tate numbers} of a given abelian $p$-adic representation of $G_K$ can be determined by using only the invariants of $K$ recoverable from the filtered profinite group $G_K^2$; this is a sharpening of the preceding result due to Mochizuki \cite[Corollary 3.1]{Mochizuki1997}.

For some of the invariants of $K$ that we use in the proof, we will give explicit \emph{group-theoretic algorithms} (cf. \cite[\S3]{Hoshi2022}) to demonstrate that those invariants can be recovered entirely from the (filtered) profinite group structure of $G_K^m$ for some $m \geq 1$ in \cref{section:2025.10.15.11.33.53,section:2025.10.15.11.34.49}.
The reader will further observe that some of those invariants can be recovered even without the filtration attached to the profinite group $G_K^m$, although the filtration is essential when we endow the $G_K^2$-module $(K^1)_+$ with the $p_K$-adic topology, which forces us to keep the additional conditions on filtration in \cref{theorem:2025.10.15.11.32.12,theorem:2025.10.15.11.32.25} (see \cref{proposition:2025.10.15.11.38.37} for more details).

\subsection*{Terminology and notation}

\noindent \emph{Sets and topological spaces}.

\begin{itemize}
    \item For a set $X$, we shall denote by $\sharp X$ the \emph{cardinality} of $X$.
    \item For a topological space $X$ and a subset $Y \subseteq X$, we shall denote by $\overline{Y}$ the \emph{closure} of $Y$ in $X$.
\end{itemize}

\noindent \emph{Groups}.

\begin{itemize}
    \item For a group $G$ and a set $X$ on which $G$ acts (on the left), we shall denote by $X^G$ the \emph{subset of $G$-invariant elements} of $X$.
\item For a group $G$ and a subset $S \subseteq G$, we shall denote by $Z_G (S)$ the \emph{centralizer} of $S$ in $G$, and write $Z (G) \coloneqq Z_G (G)$ for the \emph{center} of $G$.
    We shall say that $G$ is \emph{center-free} if $Z (G)$ is trivial.
    \item For a group $G$, we shall denote by $\widehat{G}$ (or $G^\wedge$) the \emph{profinite completion} of $G$.
    For a group homomorphism $\alpha \colon G_1 \to G_2$, we shall denote by $\widehat{\alpha}$ (or $\alpha^\wedge$) the canonical homomorphism $\widehat{G_1} \to \widehat{G_2}$ induced by $\alpha$.
\end{itemize}

\noindent \emph{Rings and modules}.

\begin{itemize}
    \item
For a ring $A$, we denote by $A_+$ (resp. $A^\times$) the \emph{additive (resp. multiplicative) group} of $A$.
    \item For an abelian group $M$ and an integer $n$, we shall denote by $M_\tor$ (resp. $M [n]$) the \emph{torsion (resp. $n$-torsion) subgroup} of $M$.
    We shall write $M_{/ \tor}$ for the module $M / M_\tor$.
\item For a field $k$, we shall denote by $\mu_n (k) = k^\times [n]$ the \emph{group of $n$\textsuperscript{th} roots of unity} in $k$.
    \item For a field $k$, we shall fix an \emph{algebraic closure} $k^\alg$ of $k$, and denote by $k^\sep \subseteq k^\alg$ the \emph{separable closure}.
\end{itemize}

\subsection*{Acknowledgement}

The author would like to express his sincere gratitude to Prof.~Yuichiro Taguchi for suggesting the topic itself and providing several perspectives through discussions.
He also thanks Naganori Yamaguchi for leaving detailed comments and valuable feedback---which eventually led to a substantial improvement regarding the uniqueness assertion of \cref{theorem:2025.10.15.11.32.25}---after carefully reading the draft.
Furthermore, the author thanks Prof.~Akio Tamagawa, Prof.~Yuichiro Hoshi, Takahiro Murotani, Kaiji Kondo, Takuya Tanaka, Shota Tsujimura and an anonymous referee for their useful suggestions and advice.

\section{Preliminary notions}
\label{section:2025.10.15.11.32.38}

\noindent \emph{Filtered profinite groups}.
Let $G$ be a profinite group.
We call a family $\left\{ G (v) \right\}_{v\in [0, +\infty)}$ of closed normal subgroups of $G$ a \emph{filtration} of $G$, if $G (v_1) \supseteq G (v_2)$ for any $v_1, v_2$ with $v_1 \leq v_2$.
We say that $G$ is a \emph{filtered profinite group} if a filtration is attached to it.

Let $G_\circ, G_\bullet$ be filtered profinite groups.
We shall say that an isomorphism $\alpha \colon G_\circ \to G_\bullet$ (of profinite groups) is an isomorphism of \emph{filtered} profinite groups if
\begin{equation*}
\alpha (G_\circ (v))
    =
        G_\bullet (v)
\end{equation*}
for all $v$; we denote by $\IsomFilt (G_\circ, G_\bullet)$ the set of isomorphisms of filtered profinite groups from $G_\circ$ into $G_\bullet$.
Note that the group $\Inn (G_\bullet)$ of \emph{inner automorphisms} of $G_\bullet$ acts on $\IsomFilt (G_\circ, G_\bullet)$, since $G_\bullet (v)$ is a normal subgroup of $G_\bullet$ for each $v$.
Hence we can define the set of \emph{filtration-preserving outer isomorphisms} from $G_\circ$ to $G_\bullet$:
\begin{equation*}
    \OutFilt (G_\circ, G_\bullet)
    \coloneqq
        \Inn (G_\bullet)
        \backslash
        \IsomFilt (G_\circ, G_\bullet).
\end{equation*}

\vspace{1em}
\noindent \emph{Ramification groups}.
For a mixed-characteristic local field $K$ and any Galois extension $F / K$ contained in $K^\alg$, the Galois group $G = \Gal (F / K)$ is a profinite group equipped with the filtration defined by the \emph{upper-numbering ramification groups} (cf. \cite[Chap. IV, \S3]{Serre1979}); we denote by $G (v)$ the $v$\textsuperscript{th} ramification group for a real number $v\geq 0$.
This filtration is compatible with quotients in the following sense: if $N$ is a closed normal subgroup of $G$, then
\begin{equation}
    \label{equation:2025.10.15.11.32.46}
(G / N) (v)
    =
        G (v) N / N
\end{equation}
for all $v\geq 0$ (cf. \emph{loc. cit.}).
Therefore, given a fundamental system $\mathscr{N}$ of neighborhoods of the identity element consisting of open normal subgroups of $G$, we have a natural isomorphism
\begin{equation}
    \label{equation:2025.10.15.11.32.55}
G (v)
    \xrightarrow{\cong}
        \varprojlim_{N \in \mathscr{N}} (G / N) (v)
\end{equation}
of profinite groups.
Note that
\begin{align*}
G (0)
    &=
        \Gal (F / (F \cap K^\un)), \\
\overline{\bigcup_{v > 0} G (v)}
    &=
        \Gal (F / (F \cap K^\tame)),
\end{align*}
where $K^\nr$ (resp. $K^\tame$) is the \emph{maximal unramified extension} (resp. \emph{maximal tamely ramified extension}) of $K$ in $K^\alg$, i.e., $G (0)$ (resp. $G (0+) \coloneqq \overline{\bigcup_{v > 0} G (v)}$) is precisely the \emph{inertia group} (resp. \emph{wild inertia group}) of $G$. (See also \emph{loc. cit.}, Exercise 1.)

Suppose that $L / K$ is a finite Galois subextension of $F / K$.
We set $H \coloneqq \Gal (F / L)$, so that $G / H = \Gal (L / K)$.
We define the function $\phi = \phi_{L / K} \colon [0, +\infty)\to [0, +\infty)$ as the inverse function of
\begin{equation*}
\psi (v)
    =
        \psi_{L / K} (v)
        \coloneqq
            \int_0^v \left(
                (G / H) : (G / H) (w)
            \right) \, dw.
\end{equation*}
It is clear from \eqref{equation:2025.10.15.11.32.46} that $\phi$ and $\psi$ are determined by the groups $H$, $G$, and $G (v)$ for $v \geq 0$.
Suppose that $N$ is an open subgroup in $H$, and that $N \unlhd G$.
One can easily verify that
\begin{equation*}
(H / N) (w)
    =
        (H / N) \cap (G / N) (\phi (w))
\end{equation*}
for all $w \geq 0$ (from, e.g., \emph{loc. cit.}, Proposition 15), and derive the following lemma from \eqref{equation:2025.10.15.11.32.55}.

\begin{lemma}
    \label{lemma:2025.10.15.11.33.03}
    For a real number $w \geq 0$, the $w$\textsuperscript{th} ramification group $H (w)$ of $H$ is determined by the groups $H$, $G$, and $G (v)$ for $v \geq 0$: we have
    \begin{equation*}
H (w)
        =
            \varprojlim_N \left\{
                (H / N) \cap (G / N) (\phi (w))
            \right\}
    \end{equation*}
    as a subset of $H = \varprojlim (H / N)$, where $N$ runs through the open subgroups of $H$ such that $N \unlhd G$.
\end{lemma}

\vspace{1em}
\noindent \emph{Solvable quotients of profinite groups}.
\label{2025.10.15.11.33.37}
For a profinite group $G$, we denote by $\overline{[G, G]}$ the closed subgroup generated by the \emph{commutators} of $G$, i.e., the elements of the form $\sigma \tau \sigma^{-1} \tau^{-1}$, where $\sigma, \tau \in G$.
We inductively define the decreasing sequence
\begin{equation*}
G
    =
        G^{[0]}
    \supseteq
        G^{[1]}
    \supseteq
        \cdots
    \supseteq
        G^{[m]}
    \supseteq
        \cdots
\end{equation*}
of closed normal subgroups of $G$, by $G^{[m + 1]} = \overline{[G^{[m]}, G^{[m]}]}$.
Note that $G^{[m]}$ are \emph{characteristic subgroups} of $G$, i.e., every automorphism of $G$ restricts to an automorphism of $G^{[m]}$.
We say that a profinite group $G$ is \emph{$m$-step solvable} (resp. \emph{abelian}) if $G^{[m]}$ (resp. $G^{[1]}$) is trivial.
We denote by $G^m$ the quotient $G / G^{[m]}$, and call it the \emph{maximal $m$-step solvable quotient} of $G$.
We will often write $G^\ab$ instead of $G^1$, and call it the \emph{maximal abelian quotient} or \emph{abelianization} of $G$.

For a field $k$, we shall denote by $k^m$ (resp. $k^\ab$) the subextension of $k^\sep / k$ fixed by $G_k^{[m]}$ (resp. $G_k^{[1]}$), and call it the \emph{maximal {$m$-step solvable} (resp. {abelian}) extension} of $k$.
In particular, we have
\begin{equation*}
G_k^m
    =
        \Gal (k^m / k), \quad
G_k^\ab
    =
        \Gal (k^\ab / k).
\end{equation*}

\begin{definition}
    Let $m$ be an integer $\geq 0$, and let $G$ be a profinite group.
    We shall say that $G$ is a \emph{profinite group of ${\MLF}$- (resp. ${\MLF}^{m}$-, resp. ${\MLF}^\ab$-) type} if there exists an isomorphism of profinite groups between $G$ and $G_K$ (resp. $G_K^{m}$, resp. $G_K^\ab$), for some mixed-characteristic local field $K$.
    We define \emph{filtered profinite groups of ${\MLF}$- (resp. ${\MLF}^{m}$-, resp. ${\MLF}^\ab$-) type} in a similar way.
\end{definition}

We prove the following lemma for later use.

\begin{lemma}
    \label{lemma:2025.10.15.11.33.30}
    Let $m, n$ be integers $\geq 0$.
    \begin{enumerate}
        \item Let $\varGamma$ be a profinite group, $H$ an open subgroup of $\varGamma^{m + n}$ containing
        \begin{equation*}
(\varGamma^{m + n})^{[m]}
            =
                \Ker (
\varGamma^{m + n}
                    \twoheadrightarrow
                        \varGamma^{m}
                )
            =
                \varGamma^{[m]} / \varGamma^{[m + n]}.
        \end{equation*}
        If we denote by $\tilde{H}$ the inverse image of $H$ under the natural surjection $\varGamma \twoheadrightarrow \varGamma^{m + n}$, then the natural surjection $\tilde{H}^{n} \twoheadrightarrow H^{n}$ is injective.
        \item Let $k$ be a field.
        For a finite extension $l / k$, we have
        \begin{equation*}
G_l^n
            =
                \Gal(k^{m + n} / l)^n
        \end{equation*}
        if $l$ is contained in $k^m$.
        In particular, if $G$ is a profinite group of $\MLF^{m + n}$-type (i.e., $G = \varGamma^{m + n}$ for some profinite group $\varGamma$ of $\MLF$-type), and $H$ is an open subgroup of $G$ containing $G^{[m]}$, then $H^n$ is a profinite group of $\MLF^n$-type.
    \end{enumerate}
\end{lemma}

\begin{proof}\hfill
    \begin{enumerate}
        \item Since the natural surjection $\tilde{H} \twoheadrightarrow H = \tilde{H} / \varGamma^{[m + n]}$ induces an isomorphism
        \begin{equation*}
\tilde{H}^{[n]} \varGamma^{[m + n]}
                /
                \varGamma^{[m + n]}
            =
                \tilde{H}^{[n]}
                /
                (\tilde{H}^{[n]} \cap \varGamma^{[m + n]})
            \xrightarrow{\cong}
                (\tilde{H} / \varGamma^{[m + n]})^{[n]},
        \end{equation*}
        we have a natural isomorphism $\tilde{H} / \tilde{H}^{[n]} \varGamma^{[m + n]} \to H^{n}$.
        It follows from the hypothesis that $\tilde{H} \supseteq \varGamma^{[m]}$ (and that $\tilde{H}^{[n]} \supseteq \varGamma^{[m + n]}$), and hence the assertion holds.
        \item Apply (1) to the case $\varGamma = G_k$, $H = \Gal (k^{m + n} / l)$.
    \end{enumerate}
\end{proof}

\begin{remark}\hfill
    \begin{enumerate}
        \item If $G$ is a profinite group of $\MLF$-type, $G$ is \emph{prosolvable} \cite[Chap. IV, Corollary 5 of Proposition 7]{Serre1979}; hence
        \begin{equation*}
            \bigcap_{m \geq 0} G^{[m]} = \{ 1 \}.
        \end{equation*}
        However, $G$ itself is \emph{not solvable}, i.e., $G^{[m]} \neq \{ 1 \}$ for every $m \geq 0$.
        This can be seen from the fact that, for every prime number $p$, the wild inertia group of $G_{\QQ_p}$ is isomorphic to a free pro-$p$ group of countably infinite rank \cite[Proposition 7.5.1]{NeukirchSchmidtWingberg2008}, which is not solvable.
        Therefore, the sequence $\left\{ G^{[m]} \right\}_{m \geq 0}$ is strictly decreasing.
        \item Let $G$ be a profinite group of $\MLF^m$-type for some integer $m \geq 0$.
        If we denote by $m (G)$ the minimal integer $n$ such that $G^{[n]} = \{ 1 \}$, then $m = m (G)$.
        In other words, $m (G)$---which is group-theoretically determined from the profinite group $G$---is the only integer $m \geq 0$ for which $G$ is a profinite group of $\MLF^m$-type: assume that $G \cong G_K^m$ for some mixed-characteristic local field $K$ and an integer $m$.
        Then obviously $G^{[m]} = \{ 1 \}$, and it is clear from (1) that $G^{[n]} \neq \{ 1 \}$ if $n < m$. Thus $m (G)$ equals $m$ by definition.
    \end{enumerate}
\end{remark}

\section{Restoration of the cyclotomic character}
\label{section:2025.10.15.11.33.53}

Let $K$ be a mixed-characteristic local field, with ring of integers $\mathscr{O}_K$, maximal ideal $\mathfrak{p}_K$ and residue field $\mathfrak{k}_K$.
In the current section, we show that the following invariants of $K$ can be recovered \emph{group-theoretically}, requiring at most the \emph{maximal $2$-step solvable quotient} $G_K^2$ of $G_K$:
\begin{itemize}
    \item $p_K$, the \emph{residue characteristic} of $K$, i.e., the characteristic of $\mathfrak{k}_K$;
    \item $\varepsilon_K$, the \emph{parity index} of $p_K$ (equal to $1$ if $p_K$ is odd, and $2$ if $p_K$ is even);
    \item $a_K$, the largest integer $\geq 0$ such that $K$ contains a $(p_K^{a_K})$\textsuperscript{th} root of unity;
    \item $d_K$, the \emph{absolute degree} $[K : \QQ_{p_K}]$ of $K$;
    \item $e_K$, the \emph{absolute ramification index} of $K$, i.e., the integer satisfying $p_K \mathscr{O}_K = \mathfrak{p}_K^{e_K}$;
    \item $f_K$, the \emph{absolute inertia degree} $[\mathfrak{k}_K : \FF_{p_K}]$ of $K$;
    \item $\chi_{\cyc, K} \colon G_K \to \widehat{\ZZ}^\times = \Aut (\varprojlim \mu_n (K^\alg))$, the \emph{cyclotomic character} of $K$;
    \item $\chi_K \colon G_K \to \ZZ_{p_K}^\times$, the \emph{$p_K$-adic cyclotomic character} of $K$.
\end{itemize}

\subsection*{Local class field theory; its application}

Suppose that $G$ is a profinite group of $\MLF^\ab$-type, i.e., there exists an isomorphism $G \to G_K^\ab = \Gal (K^\ab / K)$ of profinite groups for some mixed-characteristic local field $K$.
We first observe the structure of the group $K^\times$.
Let $\pi \in K^\times$ be a \emph{uniformizer} of $K$, i.e., an element such that $\mathfrak{p}_K = \pi \mathscr{O}_K$.
Then we have the isomorphisms of topological groups
\begin{align*}
K^\times
    &=
        U_K \cdot \pi^{\ZZ_+}
    \xrightarrow{\cong}
        U_K \oplus \ZZ_+, \\
U_K
    &=
        \mu_{\sharp (\mathfrak{k}_K) - 1} (K) \cdot U_K (1)
    \xrightarrow{\cong}
(\ZZ / (p_K^{f_K} - 1) \ZZ)_+
        \oplus
            (\ZZ / p_K^{a_K} \ZZ)_+
        \oplus
            ((\ZZ_{p_K})_+)^{\oplus d_K},
\end{align*}
where $U_K = U_K (0)$ (resp. $U_K (n)$) is the \emph{unit group} $\mathscr{O}_K^\times$ (resp. \emph{$n$\textsuperscript{th} higher unit group} $1 + \mathfrak{p}_K^n$) of $\mathscr{O}_K$ (cf. \cites[Chap. II, \S2]{Iwasawa1986}[Chap. II, \S5]{Neukirch1999}).
We recall from \emph{local class field theory} (cf., e.g., \cite{FesenkoVostokov1993,Iwasawa1986,MilneCFT,Neukirch1999,Serre1976,Serre1979,Yoshida2008}) that the \emph{local reciprocity map} (or \emph{local Artin map}) $\Art_K \colon K^\times \to G_K^\ab$ fits into the following commutative diagram (in which the rows are exact and split)
\begin{equation*}
    \begin{tikzcd}
        1 & {U_K} & {K^\times} & {\ZZ_+} & 1 \\
        1 & {G_K^\ab(0)} & {G_K^\ab} & {\Gal(K^\nr/K)} & 1
        \arrow[from=1-1, to=1-2]
        \arrow[from=1-2, to=1-3]
        \arrow["{\cong, \Art_K\vert_{U_K}}", from=1-2, to=2-2]
        \arrow["{{{\ord_K}}}", from=1-3, to=1-4]
        \arrow["{{{\Art_K}}}", from=1-3, to=2-3]
        \arrow[from=1-4, to=1-5]
        \arrow["{{{\Frob_K^{(-)}}}}", from=1-4, to=2-4]
        \arrow[from=2-1, to=2-2]
        \arrow[from=2-2, to=2-3]
        \arrow[from=2-3, to=2-4]
        \arrow[from=2-4, to=2-5]
    \end{tikzcd},
\end{equation*}
where $\ord_K$ is the \emph{normalized discrete valuation} on $K$ and $\Frob_K$ is the \emph{arithmetic Frobenius} of $K$.
Hence it yields an isomorphism of profinite groups
\label{2025.10.15.11.34.04}
\begin{equation*}
\widehat{K^\times} \, (
            \cong
                U_K \oplus \widehat{\ZZ}_+
        ) \,\,\,
    \xrightarrow{\cong} \,\,\,
        G_K^\ab \, (
            \cong
                G_K^\ab (0) \oplus \Gal (K^\nr / K)
        ),
\end{equation*}
by profinite completion.
In particular, we have
\begin{equation}
    \label{equation:2025.10.15.11.34.10}
G
    \cong
        G_K^\ab
    \cong
(\ZZ / (p_K^{f_K} - 1) \ZZ)_+
        \oplus
            (\ZZ / p_K^{a_K} \ZZ)_+
        \oplus
            ((\ZZ_{p_K})_+)^{\oplus d_K}
        \oplus
            \widehat{\ZZ}_+
\end{equation}
as profinite groups.
We denote by $p (G)$ the uniquely determined prime number $\ell$ such that
\begin{equation*}
\log_\ell \sharp (G_{/ \tor} / \ell \cdot G_{/ \tor})
    \geq
        2.
\end{equation*}
Furthermore, we set:
\begin{itemize}
    \item $\varepsilon (G) \coloneqq 1$ (resp. $\varepsilon (G) \coloneqq 2$) if $p (G)$ is odd (resp. even);
    \item $a (G) \coloneqq \log_{p (G)} \sharp ((G_{\tor})^{(p (G))})$;
    \item $d (G) \coloneqq \log_{p (G)} \sharp (G_{/ \tor} / p (G) \cdot G_{/ \tor}) - 1$;
    \item $f (G) \coloneqq \log_{p (G)}(\sharp ((G_\tor)^{(p (G)')}) + 1)$;
    \item $e (G) \coloneqq d (G) / f (G)$,
\end{itemize}
where $(G_{\tor})^{(p (G))}$ (resp. $(G_{\tor})^{(p (G)')}$) denotes the \emph{$p (G)$-Sylow subgroup of $G_\tor$} (resp. the \emph{quotient of $G_\tor$ by its $p (G)$-Sylow subgroup}).

\begin{proposition}
    \label{proposition:2025.10.15.11.34.21}
    Let $K$ be a mixed-characteristic local field.
    Then we have
    \begin{equation*}
p_K
        =
            p (G_K^\ab), \quad
\varepsilon_K
        =
            \varepsilon (G_K^\ab), \quad
a_K
        =
            a (G_K^\ab),
    \end{equation*}
    \begin{equation*}
d_K
        =
            d (G_K^\ab), \quad
e_K
        =
            e (G_K^\ab), \quad
f_K
        =
            f (G_K^\ab).
    \end{equation*}
\end{proposition}

Intuitively speaking, $p_K$, $\varepsilon_K$, $a_K$, $d_K$, $e_K$, and $f_K$ can be recovered entirely group-theoretically from the profinite group $G_K^\ab$.

\begin{proof}
    It follows from \eqref{equation:2025.10.15.11.34.10} that $p_K$ is the only prime number $\ell$ such that
    \begin{equation*}
\log_\ell \sharp ((G_K^\ab)_{/ \tor} / \ell \cdot (G_K^\ab)_{/ \tor})
        \geq
            2.
    \end{equation*}
    Hence $p_K = p (G_K^\ab)$, $\varepsilon_K = \varepsilon (G_K^\ab)$ and
    \begin{equation*}
d (G_K^\ab)
        =
            \log_{p_K} \sharp ((G_K^\ab)_{/ \tor} / p_K \cdot (G_K^\ab)_{/ \tor}) - 1
        =
            d_K.
    \end{equation*}
    We also see from \eqref{equation:2025.10.15.11.34.10} that the pro-prime-to-$p_K$ (resp. pro-$p_K$) part of $(G_K^\ab)_\tor$ has exactly $p_K^{f_K} - 1$ (resp. $p_K^{a_K}$) elements.
    Therefore, we obtain the third, fifth and sixth equalities.
\end{proof}

\subsection*{Restoration of the cyclotomic character}

Next, we give a reconstruction algorithm which takes as input a profinite group of $\MLF^2$-type, say, $G$, and returns (the isomorphism class of) a $\ZZ_{\ell}$-representation of $G$ of rank $1$, for each prime number $\ell$.
Suppose that there exists an isomorphism $\alpha \colon G \to G_K^2$ of profinite groups for a mixed-characteristic local field $K$.
We start by choosing a decreasing sequence
\begin{equation*}
G
    =
        H_0
    \supseteq
        H_1
    \supseteq
        \cdots
    \supseteq
        H_\nu
    \supseteq
        \cdots
\end{equation*}
of open normal subgroups of $G$ such that, for each $\nu \in \ZZ_{\geq 0}$,
\begin{enumerate}[label=(\roman*)]
    \item $H_\nu^\ab [\ell^\nu] \cong (\ZZ / \ell^\nu \ZZ)_+$;
    \item $G / H_\nu$ is abelian.
\end{enumerate}
(Note that $G$ acts on $H_\nu^\ab[\ell^\nu]$ by conjugation.)
Such a sequence $\left\{ H_\nu \right\}_{\nu}$ exists: we can choose $H_\nu = \alpha^{-1}(\Gal (K^2 / K (\zeta_{\ell^\nu})))$, where $\zeta_{\ell^\nu}$ is a primitive $(\ell^\nu)$\textsuperscript{th} root of unity.
It follows immediately that $\left\{ H_\nu \right\}_{\nu}$ satisfies the condition (ii).
We can also verify that $\left\{ H_\nu \right\}_{\nu}$ satisfies the condition (i), using the local reciprocity map
\begin{equation*}
    \Art_{K (\zeta_{\ell^\nu})}
    \colon
        K (\zeta_{\ell^\nu})^\times
    \to
        G_{K (\zeta_{\ell^\nu})}^\ab
\end{equation*}
and the fact that
\begin{equation*}
H_\nu^\ab
    \cong
        \Gal (K^2 / K (\zeta_{\ell^\nu}))^\ab
    =
        G_{K (\zeta_{\ell^\nu})}^\ab,
\end{equation*}
which follows from \cref{lemma:2025.10.15.11.33.30}.

We know from local class field theory that if $L_\square \subseteq K^\ab$ is the field fixed by $\alpha (H_\square)$ for each $\square \in \left\{ \nu, \nu + 1 \right\}$, then the diagram
\begin{equation*}
    \begin{tikzcd}[column sep=3.5em]
        {H_\nu^\ab} & {\Gal(K^2/L_\nu)^\ab = G_{L_\nu}^\ab} & {L_\nu^\times} \\
        {H_{\nu+1}^\ab} & {\Gal(K^2/L_{\nu+1})^\ab = G_{L_{\nu+1}}^\ab} & {L_{\nu+1}^\times}
        \arrow["{\cong,\alpha_\nu}", from=1-1, to=1-2]
        \arrow["\Ver", from=1-1, to=2-1]
        \arrow["\Ver", from=1-2, to=2-2]
        \arrow["{\Art_{L_{\nu}}}"', from=1-3, to=1-2]
        \arrow["\subseteq", from=1-3, to=2-3]
        \arrow["{\cong,\alpha_{\nu+1}}", from=2-1, to=2-2]
        \arrow["{\Art_{L_{\nu+1}}}"', from=2-3, to=2-2]
    \end{tikzcd}
\end{equation*}
commutes, where $\Ver$ is the \emph{transfer} map (cf., e.g., \cites[Chap. VII, \S8]{Serre1979}[\S6.7]{Weibel1994}) and $\alpha_\square$ is the isomorphism of profinite groups induced by $\alpha$.
Moreover, $\Art_{L_\square}$ restricts to an isomorphism $U_{L_\square} \to G_{L_\square}^\ab (0)$, and hence to an isomorphism $\mu_{\ell^{\square}} (L_{\square}) \to G_{L_\square}^\ab[\ell^\square]$.
Therefore, $\Ver$ restricts to an injective homomorphism $H_{\nu}^\ab[\ell^\nu] \to H_{\nu+1}^\ab[\ell^{\nu+1}]$;
we identify $H_\nu^\ab[\ell^\nu]$ with a subgroup of $H_{\nu+1}^\ab[\ell^{\nu+1}]$ via $\Ver$.
(We will see later that the transfer map $\Ver$ here is in fact an injective homomorphism---cf. \cref{lemma:2025.10.15.11.37.39} (2).)

We have the inverse system
\begin{equation*}
    \begin{tikzcd}
        \cdots & {H_{\nu+1}^\ab[\ell^{\nu+1}]} & {H_\nu^\ab[\ell^\nu]} & \cdots & {H_1^\ab[\ell]}
        \arrow["{{{(-)^\ell}}}", from=1-1, to=1-2]
        \arrow["{{{(-)^\ell}}}", from=1-2, to=1-3]
        \arrow["{{{(-)^\ell}}}", from=1-3, to=1-4]
        \arrow["{{{(-)^\ell}}}", from=1-4, to=1-5]
    \end{tikzcd}
\end{equation*}
of $G$-modules induced by the homomorphisms $H_{\nu + 1}^\ab \xrightarrow{(-)^\ell} H_{\nu + 1}^\ab$.
By passage to the limit, we obtain
\begin{equation*}
    T_\ell (G) \coloneqq \varprojlim_\nu H_\nu^\ab[\ell^\nu].
\end{equation*}
It will be implicitly shown in the proof of \cref{proposition:2025.10.15.11.34.35} that the isomorphism class of the $G$-module $H_\nu^\ab[\ell^\nu]$ for each $\nu$ (and hence the isomorphism class of $T_\ell (G)$) does not depend on the choice of $H_\nu$.
We shall write
\begin{equation*}
    \chi^{(\ell)} (G)
    \colon
        G
    \to
        \Aut (T_\ell (G)) \, (
            =
                \ZZ_\ell^\times
        )
\end{equation*}
for the $\ell$-adic character of $G$ attached to $T_\ell (G)$, and we define $\chi (G)$ as follows:
\begin{equation*}
    \chi (G)
    \coloneqq
        \chi^{(p (G^\ab))} (G).
\end{equation*}

\begin{proposition}
    \label{proposition:2025.10.15.11.34.35}
    Let $K$ be a mixed-characteristic local field.
    \begin{enumerate}
        \item For each prime number $\ell$, there exists an isomorphism
        \begin{equation*}
\ZZ_{\ell} (1)
            \xrightarrow{\cong}
                T_\ell (G_K^2)
        \end{equation*}
        of $G_K^2$-modules, where $\ZZ_{\ell} (1)$ denotes the \emph{first Tate twist} of the trivial $G_K^{2}$-module $\ZZ_{\ell}$.
        \item The cyclotomic character $\chi_K$ factors through $\chi (G_K^2)$.
    \end{enumerate}
\end{proposition}

Intuitively speaking, $\chi_{\cycl, K} \colon G_K \to \widehat{\ZZ}^\times$ and $\chi_K$ can be recovered entirely group-theoretically from the profinite group $G_K^2$.

\begin{proof}
    (1) We take a decreasing sequence
    \begin{equation*}
G_K^2
        =
            H_{K, 0}
        \supseteq
            H_{K, 1}
        \supseteq
            \cdots
        \supseteq
            H_{K, \nu}
        \supseteq
            \cdots
    \end{equation*}
    of open normal subgroups of $G_K^2$ satisfying the above conditions (i) and (ii).
    We shall write $L_\nu$ for the corresponding fixed field $(K^2)^{H_{K, \nu}}$ of $H_{K, \nu}$.
    By \cref{lemma:2025.10.15.11.33.30} and the condition (ii), we have $G_{L_\nu}^\ab = H_{K, \nu}^\ab$ for each $\nu$, and thus we have a group homomorphism
    \begin{equation*}
        r_\nu
        \coloneqq
            \Art_{L_\nu}
            \colon
                L_\nu^\times
            \to
                H_{K, \nu}^\ab.
    \end{equation*}
    It is implied by the condition (i) that $L_\nu$ contains the $(\ell^\nu)$\textsuperscript{th} roots of unity.
    Moreover, it can be seen from local class field theory that $r_\nu$ respects the $G_K^2$-action (cf. \cite[Chap. IV, (4.2)]{FesenkoVostokov1993}).
    We obtain by restriction the $G_K^2$-module isomorphism
    \begin{equation*}
        r_\nu
        \colon \,\,\,
            (
(\ZZ/\ell^\nu\ZZ)_+
                \cong
            ) \,
            \mu_{\ell^\nu} (L_\nu) \,\,\,
        \xrightarrow{\cong} \,\,\,
            (
(\ZZ/\ell^\nu\ZZ)_+
                \cong
            ) \,
            H_{K, \nu}^\ab [\ell^\nu],
    \end{equation*}
    and the commutative diagram
    \begin{equation*}
        \begin{tikzcd}
            {{L_{\nu+1}^\times}} & {{L_{\nu+1}^\times}} \\
            {H_{K, \nu+1}^\ab} & {H_{K, \nu+1}^\ab}
            \arrow["{(-)^{\ell}}", from=1-1, to=1-2]
            \arrow["r_{\nu+1}", from=1-1, to=2-1]
            \arrow["r_{\nu+1}", from=1-2, to=2-2]
            \arrow["{(-)^{\ell}}", from=2-1, to=2-2]
        \end{tikzcd}
    \end{equation*}
    of $G_K^2$-modules.
    We also know from local class field theory that the diagram
    \begin{equation*}
        \begin{tikzcd}
            {{L_\nu^\times}} & {{L_{\nu+1}^\times}} \\
            {H_{K, \nu}^\ab} & {H_{K, \nu+1}^\ab}
            \arrow["\subseteq", from=1-1, to=1-2]
            \arrow["r_\nu", from=1-1, to=2-1]
            \arrow["r_{\nu+1}", from=1-2, to=2-2]
            \arrow["\Ver", from=2-1, to=2-2]
        \end{tikzcd}
    \end{equation*}
    commutes.
    Hence we have the following commutative diagram:
    \begin{equation*}
        \begin{tikzcd}
            \cdots & {\mu_{\ell^{\nu+1}}( L_{\nu+1})} & {\mu_{\ell^\nu}( {L_{\nu}})} & \cdots & {\mu_{\ell}( {L_{1}})} \\
            \cdots & {H_{K, \nu+1}^\ab[\ell^{\nu+1}]} & {H_{K, \nu}^\ab[\ell^\nu]} & \cdots & {H_{K,1}^\ab[\ell]}
            \arrow["{(-)^{\ell}}", from=1-2, to=1-3]
            \arrow["{(-)^{\ell}}", from=2-2, to=2-3]
            \arrow["{\cong,r_{\nu+1}}", from=1-2, to=2-2]
            \arrow["{\cong,r_{\nu}}", from=1-3, to=2-3]
            \arrow["{(-)^{\ell}}", from=1-1, to=1-2]
            \arrow["{(-)^{\ell}}", from=2-1, to=2-2]
            \arrow["{(-)^{\ell}}", from=1-3, to=1-4]
            \arrow["{(-)^{\ell}}", from=2-3, to=2-4]
            \arrow["{(-)^{\ell}}", from=1-4, to=1-5]
            \arrow["{(-)^{\ell}}", from=2-4, to=2-5]
            \arrow["{\cong,r_{1}}", from=1-5, to=2-5]
        \end{tikzcd}.
    \end{equation*}
    By passage to the limit, we obtain the following isomorphism of $G_K^2$-modules.
    \begin{equation*}
        r
        \coloneqq
            \varprojlim_\nu r_\nu
            \colon \,\,\,
\ZZ_{\ell} (1)
                =
                    \varprojlim_\nu \mu_{\ell^\nu} (L_\nu) \,\,\,
            \xrightarrow{\cong} \,\,\,
T_\ell (G_K^2)
                =
                    \varprojlim_\nu H_{K, \nu}^\ab[\ell^\nu]
    \end{equation*}
    
    (2) It is clear from (1) and \cref{proposition:2025.10.15.11.34.21}.
\end{proof}

\begin{remark}
    \cref{proposition:2025.10.15.11.34.35} can be considered as a local analogue of \cite[Proposition A.9]{SaidiTamagawa2022}.
\end{remark}
 
\section{Ramification groups}
\label{section:2025.10.15.11.34.49}

We keep the notation and hypotheses of \cref{section:2025.10.15.11.33.53}.
In this section, we recover the $G_K$-module structure of $(K^\ab)_+$ and its $p_K$-adic completion from the \emph{filtered} profinite group $G_K^2$.

Assume that $G$ is a profinite group of ${\MLF}^{m + 1}$-type for an integer $m \geq 1$.
It follows directly from \cref{lemma:2025.10.15.11.33.30} that if $H$ is an open subgroup of $G$ containing $G^{[m]}$, then $H^\ab$ is a profinite group of $\MLF^\ab$-type.
We denote by $I (G)$ the intersection of open subgroups $H$ such that $H \supseteq G^{[1]}$ and $e (H^\ab) = e (G^\ab)$.
We also denote by $P (G)$ the (necessarily unique) pro-$p (G^\ab)$-Sylow subgroup of $I (G)$.

\begin{lemma}
    \label{lemma:2025.10.15.11.34.55}
    Let $K$ be a mixed-characteristic local field, and let $m$ be an integer $\geq 1$.
    Then the \emph{inertia group} (resp. \emph{wild inertia group}) of $G_K^{m + 1}$ equals $I (G_K^{m + 1})$ (resp. $P (G_K^{m + 1})$).
    In particular, the inertia group (resp. wild inertia group) of $G_K^{m + 1}$ can be determined entirely group-theoretically, without the additional information on filtration.
\end{lemma}

\begin{proof}
    Keeping in mind that \emph{every unramified extension of $K$ is abelian}, one checks by using \cref{lemma:2025.10.15.11.33.30} and \cref{proposition:2025.10.15.11.34.21} that the open subgroups of $G_K^{m + 1}$ containing $I (G_K^{m + 1})$ are precisely the ones corresponding to the finite unramified extensions over $K$; hence $I (G_K^{m + 1})$ equals the inertia group.
    Since the wild inertia group is nothing but the unique pro-$p_K$-Sylow subgroup of the inertia group
    (For the finite order case, see \cite[Chap. IV, Corollaries 1 and 3 of Proposition 7]{Serre1979}.
    One easily reduces to this case, since wild inertia groups are compatible with quotients, cf. \emph{loc. cit.}, Exercise 1 of \S2.),
    the assertion on the wild inertia group holds as well.
\end{proof}

\begin{remark}
    If $m$ is an integer $\geq 2$ and $H$ is an open subgroup of $G$ containing $G^{[2]}$, $H^\ab$ is a profinite group of $\MLF^{\ab}$-type as remarked above.
    Hence in the case $m \geq 2$, one could alternatively define $P (G)$ as the intersection of open subgroups $H$ such that $H \supseteq G^{[2]}$ and $e (H^\ab) / e (G^\ab)$ is coprime to $p (G^\ab)$: keeping in mind that \emph{every tamely ramified extension of $K$ is metabelian}, one checks as in the above proof that the open subgroups of $G_K^{m + 1}$ containing $P (G_K^{m + 1})$ are precisely the ones corresponding to the finite tamely ramified extensions over $K$.
    Thus $P (G_K^{m + 1})$ equals the wild inertia group of $G_K^{m + 1}$.
\end{remark}

Once again, let $G$ be a profinite group of ${\MLF}^{m + 1}$-type for an integer $m \geq 1$, and let $\mathscr{H}_m (G)$ denote the set of open normal subgroups of $G$ containing $G^{[m]}$, ordered by reverse inclusion.
For each $H \in \mathscr{H}_m (G)$, we denote by $U (H)$ the image of $H \cap P (G)$ under the natural map $H \twoheadrightarrow H^\ab$, then we see that $G$ acts on $U (H)$ by conjugation.

We first claim that, for $H_1, H_2 \in \mathscr{H}_m (G)$ with $H_1\supseteq H_2$, the transfer map $\Ver \colon H_1^\ab \to H_2^\ab$ restricts to $U (H_1) \to U (H_2)$, and that $\left\{ U (H) \right\}_{H \in \mathscr{H}_m (G)}$ forms a direct system of $G$-modules, together with $V_{1,2} \coloneqq \Ver \vert_{U (H_1)} \colon U (H_1) \to U (H_2)$ for each pair $H_1 \supseteq H_2$.
Suppose that there exists an isomorphism $\alpha \colon G \to G_K^{m + 1}$ of profinite groups for some mixed-characteristic local field $K$.
Then, for each $\square \in \left\{ 1, 2 \right\}$, the image of $H_\square$ equals $\Gal (K^{m + 1} / L_\square)$ for a finite Galois subextension $L_\square / K$ of $K^m / K$.
Note that $\Gal (K^{m + 1} / L_\square)^\ab = G_{L_\square}^\ab$ by \cref{lemma:2025.10.15.11.33.30} (and hence $H_\square^\ab$ is of $\MLF^\ab$-type).
The isomorphism $\alpha_\square \colon H_\square^\ab \to \Gal (K^{m + 1} / L_\square)^\ab$ induced by $\alpha$ indeed fits into the following commutative diagram:
\begin{equation}
    \label{equation:2025.10.17.23.48.58}
    \begin{tikzcd}
        {H^\ab_1} & {\Gal(K^{m + 1}/L_1)^\ab = G_{L_1}^\ab} & {L_1^\times} \\
        {H^\ab_2} & {\Gal(K^{m + 1}/L_2)^\ab = G_{L_2}^\ab} & {L_2^\times}
        \arrow["\Ver", from=1-1, to=2-1]
        \arrow["\Ver", from=1-2, to=2-2]
        \arrow["{{\cong,\alpha_1}}", from=1-1, to=1-2]
        \arrow["{{\cong,\alpha_2}}", from=2-1, to=2-2]
        \arrow["\subseteq", from=1-3, to=2-3]
        \arrow["{\Art_{L_1}}"', from=1-3, to=1-2]
        \arrow["{\Art_{L_2}}"', from=2-3, to=2-2]
    \end{tikzcd}
\end{equation}
We see from \cref{lemma:2025.10.15.11.34.55} that $H_\square \cap P (G) \subseteq G$ is mapped onto
\begin{equation*}
\Gal (K^{m + 1} / L_\square) \cap P (G_K^{m + 1})
    =
        \Gal (K^{m + 1} / L_\square) \cap G_K^{m + 1} (0+)
    =
        \Gal (K^{m + 1} / L_\square) (0+)
\end{equation*}
under $\alpha$; hence $U (H_\square) \subseteq H_\square^\ab$ is mapped onto $G_{L_\square}^\ab (0+)$ under $\alpha_\square$.
Therefore, it suffices to show that the middle vertical arrow restricts to $G_{L_1}^\ab (0+) \to G_{L_2}^\ab (0+)$.
But by local class field theory and the theorem of Hasse-Arf (cf. \cite[Chap. V]{Serre1979}), $U_{L_\square} (1)$ is mapped onto $G_{L_\square}^\ab (0+)$ under the local reciprocity map $\Art_{L_\square}$, and hence
\begin{equation*}
\Ver (G_{L_1}^\ab (0+))
    =
        \Ver (\Art_{L_1} (U_{L_1} (1)))
    \subseteq
        \Art_{L_2} (U_{L_2} (1))
    =
        G_{L_2}^\ab (0+).
\end{equation*}
In particular, the restriction of $\Art_{L_\square}$ to $U_{L_\square} (1) \to G_{L_{\square}}^\ab (0+)$ is an isomorphism.
It follows immediately that $\left\{ U (H) \right\}_{H \in \mathscr{H}_m (G)}$ is a direct system induced by the direct system $\left\{ U_L (1) \right\}_{L / K}$, where $L / K$ runs through the finite Galois subextensions of $K^m / K$; each $U (H)$ is a (topological) $\ZZ_p$-module of finite rank, where $p \coloneqq p (G^\ab) = p_K$.
Hence we obtain a direct system $\left\{ U (H) \otimes_{\ZZ_p} \QQ_p \right\}_{H \in \mathscr{H}_m (G)}$ of $G$-modules;
we set
\begin{equation*}
    (k^m)_+ (G)
    \coloneqq
        \varinjlim_{H \in \mathscr{H}_m (G)} \left(
            U (H) \otimes_{\ZZ_{p}} \QQ_{p}
        \right).
\end{equation*}

\begin{proposition}
    \label{proposition:2025.10.15.11.35.03}
    Let $K$ be a mixed-characteristic local field, and let $m$ be an integer $\geq 1$.
    Then there exists an isomorphism
    \begin{equation*}
(k^m)_+ (G_K^{m + 1})
        \xrightarrow{\cong}
            (K^m)_+
    \end{equation*}
    of $G_K^{m + 1}$-modules.
\end{proposition}

Speaking from an intuitive level, the $G_K^{m + 1}$-module $(K^m)_+$ can be recovered entirely group-theoretically from the profinite group $G_K^{m + 1}$.

\begin{proof}
    Let $H_\square = \Gal (K^{m + 1} / L_\square) \in \mathscr{H}_m (G_K^{m + 1})$ for each $\square \in \left\{ 1, 2 \right\}$, and assume that $H_1 \supseteq H_2$.
    By construction, $U (H_\square) = \Gal (K^{m + 1} / L_\square)^\ab (0+) = G_{L_\square}^\ab (0+)$.
    The $p_K$-adic logarithm (cf. \cites[Chap. II, \S5]{Neukirch1999}[Chap. IV, \S2]{Koblitz1984}) gives the commutative diagram
    \begin{equation}
        \label{equation:2025.10.15.11.35.10}
        \begin{tikzcd}[column sep=3.5em]
            {G_{L_1}^\ab(0+)\otimes_{\ZZ_{p_K}} \QQ_{p_K}} & {U_{L_1}(1)\otimes_{\ZZ_{p_K}} \QQ_{p_K}} & {(L_1)_+} \\
            {G_{L_2}^\ab(0+)\otimes_{\ZZ_{p_K}} \QQ_{p_K}} & {U_{L_2}(1)\otimes_{\ZZ_{p_K}} \QQ_{p_K}} & {(L_2)_+}
            \arrow["{{\cong,\Art_{L_1}^{-1}}}", from=1-1, to=1-2]
            \arrow["{{{{V_{1,2}\otimes \id}}}}", from=1-1, to=2-1]
            \arrow["{{{{\cong, \log}}}}", from=1-2, to=1-3]
            \arrow["\subseteq\otimes\id", from=1-2, to=2-2]
            \arrow["\subseteq", from=1-3, to=2-3]
            \arrow["{{\cong,\Art_{L_2}^{-1}}}", from=2-1, to=2-2]
            \arrow["{{{{\cong,\log}}}}", from=2-2, to=2-3]
        \end{tikzcd}
    \end{equation}
    of $G_K^{m + 1}$-modules.
    By passage to the limit, we obtain the desired isomorphism
    \begin{equation*}
(k^m)_+ (G_K^{m + 1})
            =
                \varinjlim_{L / K} \left(
                    G_{L}^\ab (0+) \otimes_{\ZZ_{p_K}} \QQ_{p_K}
                \right) \,\,\,
        \xrightarrow{\cong} \,\,\,
(K^m)_+
            =
                \varinjlim_{L / K} L_+,
    \end{equation*}
    where $L / K$ runs through the finite Galois subextensions of $K^{m} / K$.
\end{proof}

For an infinite algebraic extension $F / K$, we shall denote by $\mathscr{C}_F$ or $\mathscr{C} (F)$ the $p_K$-adic completion of $F$, and often write $\CC_{p_K}$ instead of $\mathscr{C}_{K^\alg}$.
In \cite[Proposition 2.2]{Mochizuki1997}, it is shown that the isomorphism class of the $G_K$-module $(\CC_{p_K})_+$ can be recovered group-theoretically from the \emph{filtered} profinite group $G_K$; we will prove an ``$m$-step solvable analogue'' of this result.
To do so, we further assume that $G$ is a \emph{filtered} profinite group of $\MLF^{m + 1}$-type, and that $\alpha$ is an isomorphism of \emph{filtered} profinite groups.

For any closed normal subgroup $N$ of $G$, we shall equip $G / N$ with the filtration defined by
\begin{equation*}
    (G / N) (v)
    \coloneqq
        G (v) N / N
\end{equation*}
for each $v \geq 0$.

Suppose that $H \in \mathscr{H}_m (G)$.
Then there exists a finite Galois subextension $L / K$ of $K^m / K$ such that $\alpha (H) = \Gal (K^{m + 1} / L)$.
We define the functions $\phi_H, \psi_H \colon [0, +\infty) \to [0, +\infty)$ by
\begin{align*}
\psi_H (v)
    &\coloneqq
        \int_0^v (
            (G / H) : (G / H) (w)
        ) \, dw, \\
\phi_H (w)
    &\coloneqq
        \psi_H^{-1} (w).
\end{align*}
We regard $H$ as a filtered profinite group by setting
\begin{equation*}
    H (w)
    \coloneqq
        \varprojlim_N \left\{
            (H / N) \cap (G / N) (\phi_H (w))
        \right\} \quad
        \left(
            \subseteq
                H
            =
                \varprojlim_N (H / N)
        \right)
\end{equation*}
for each $w \geq 0$, where $N$ runs through the open subgroups of $H$ such that $N \unlhd G$.
As a direct consequence of \cref{lemma:2025.10.15.11.33.03}, $\alpha \vert_H \colon H \to \Gal (K^{m + 1} / L)$ is an isomorphism of \emph{filtered} profinite groups.

We denote by $U (H, w)$ the image of $H (w)$ under the natural map $H \twoheadrightarrow H^\ab$.
Then we see that $G$ acts on $U (H, w)$ by conjugation.
We claim that, for $H_1, H_2 \in \mathscr{H}_m (G)$ with $H_1 \supseteq H_2$, the transfer map $\Ver \colon H_1^\ab \to H_2^\ab$ restricts to
\begin{equation*}
U (H_1, \varepsilon (G) e (H_1^\ab))
    \to
        U (H_2, \varepsilon (G) e (H_2^\ab)),
\end{equation*}
and that if we denote by $U' (H)$ the group $U (H, \varepsilon (G) e (H^\ab))$ for each $H \in \mathscr{H}_m (G)$,
\begin{equation*}
    \left\{ U' (H) \right\}_{H \in \mathscr{H}_m (G)}
\end{equation*}
forms a direct system of $G$-modules, together with $V'_{1, 2} \coloneqq \Ver \vert_{U' (H_1)} \colon U' (H_1) \to U' (H_2)$ for each pair $H_1 \supseteq H_2$.
As we have seen above, $H_\square (w)$ is mapped onto $\Gal (K^{m + 1} / L_\square) (w)$ under $\alpha \vert_{H_\square}$ for all $w \geq 0$; thus $U' (H_\square) \subseteq H_\square^\ab$ is mapped onto $G_{L_\square}^\ab (\varepsilon (G) e (H_\square^\ab))$ under $\alpha_\square$.
Therefore, in order to prove the claim, it suffices to show that the middle vertical arrow of \eqref{equation:2025.10.17.23.48.58} restricts to
\begin{equation*}
G_{L_1}^\ab (\varepsilon_K e_{L_1})
    \to
        G_{L_2}^\ab (\varepsilon_K e_{L_2}).
\end{equation*}
We have
\begin{equation*}
\mathfrak{p}_{L_1}^{\varepsilon_K e_{L_1}}
    \subseteq
        \left(
            \mathfrak{p}_{L_1} \mathscr{O}_{L_2}
        \right)^{\varepsilon_K e_{L_1}}
    =
        \left(
            \mathfrak{p}_{L_2}^{e_{L_2} / e_{L_1}}
        \right)^{\varepsilon_K e_{L_1}}
    =
        \mathfrak{p}_{L_2}^{\varepsilon_K e_{L_2}},
\end{equation*}
and it follows that $U_{L_1} (\varepsilon_K e_{L_1}) \subseteq U_{L_2} (\varepsilon_K e_{L_2})$.
Together with the fact that $\Art_{L_\square}$ restricts to an isomorphism $U_{L_\square} (w) \to G_{L_\square}^\ab (w)$ for all $w \geq 0$ \cite[p.155, Theorem 1]{Serre1976}, we conclude that
\begin{equation*}
\Ver (G_{L_1}^\ab (\varepsilon_K e_{L_1}))
    =
        \Ver (\Art_{L_1} (U_{L_1} (\varepsilon_K e_{L_1})))
    \subseteq
        \Art_{L_2} (U_{L_2} (\varepsilon_K e_{L_2}))
    =
        G_{L_2}^\ab (\varepsilon_K e_{L_2}).
\end{equation*}
Therefore, we obtain a direct system $\left\{ U' (H) \right\}_{H \in \mathscr{H}_m (G)}$ of $G$-modules in a way similar to the way in which we obtained $\left\{ U (H) \right\}_{H \in \mathscr{H}_m (G)}$.
We again put $p \coloneqq p (G^\ab) = p_K$.
Note that, for each $H \in \mathscr{H}_m (G)$ and the subextension $L / K$ fixed by $\alpha (H)$, the natural map $U_L (\varepsilon_K e_L) \to U_L (1) \otimes_{\ZZ_p} \QQ_p$ (and hence the natural map $U' (H) \to U (H) \otimes_{\ZZ_p} \QQ_p$) is injective since $U_L (\varepsilon_K e_L)$ is contained in the non-torsion part of $U_L$ (as $\varepsilon_K e_L > e_L / (p - 1)$ and the $p$-adic logarithm restricts to an isomorphism $U_L (\varepsilon_K e_L) \to \mathfrak{p}_L^{\varepsilon_K e_L}$).
We identify $U' (H)$ with a submodule of $U (H) \otimes_{\ZZ_p} \QQ_p$;
we set
\begin{align*}
    (\mathscr{O}_{k^m})_+ (G)
    &\coloneqq
        p^{-\varepsilon (G)} \left(
            \varinjlim_{H \in \mathscr{H}_m (G)} U' (H)
        \right) \quad
        \left(
            \subseteq
                (k^m)_+ (G)
            =
                \varinjlim_{H \in \mathscr{H}_m (G)} \left(
                    U (H) \otimes_{\ZZ_p} \QQ_p
                \right)
        \right), \\
(\mathscr{C}_{k^m})_+ (G)
    &\coloneqq
        \left(
            \varprojlim_{n} \left(
                (\mathscr{O}_{k^m})_+ (G) / p^n \cdot (\mathscr{O}_{k^m})_+ (G)
            \right)
        \right) \otimes_{\ZZ_p} \QQ_p.
\end{align*}

\begin{proposition}
    \label{proposition:2025.10.15.11.38.37}
    Let $K$ be a mixed-characteristic local field, and let $m$ be an integer $\geq 1$.
    Then the isomorphism of \cref{proposition:2025.10.15.11.35.03} restricts to an isomorphism
    \begin{equation*}
(\mathscr{O}_{k^m})_+ (G_K^{m + 1})
        \xrightarrow{\cong}
            (\mathscr{O}_{K^m})_+
    \end{equation*}
    of $G_K^{m + 1}$-modules, where $\mathscr{O}_{K^m}$ denotes the integral closure of $\mathscr{O}_K$ in $K^m$.
    In particular, there exists an isomorphism
    \begin{equation*}
(\mathscr{C}_{k^m})_+ (G_K^{m + 1})
        \xrightarrow{\cong}
            (\mathscr{C}_{K^m})_+
    \end{equation*}
    of $G_K^{m + 1}$-modules.
\end{proposition}

Speaking from an intuitive level, the $G_K^{m + 1}$-module $(\mathscr{C}_{K^m})_+$ can be recovered entirely group-theoretically from the \emph{filtered} profinite group $G_K^{m + 1}$.

\begin{proof}
    Let $H_\square = \Gal (K^{m + 1} / L_\square) \in \mathscr{H}_m (G_K^{m + 1})$ for each $\square \in \left\{ 1, 2 \right\}$, and assume that $H_1 \supseteq H_2$.
    By construction, we have $U' (H_\square) = G_{L_\square}^\ab(\varepsilon_K e_{L_\square})$.
    Under the isomorphism $\Art_{L_\square}^{-1} \colon G_{L_\square}^\ab (0+) \to U_{L_\square} (1)$, the subgroup $G_{L_\square}^\ab (\varepsilon_K e_{L_\square})$ is mapped onto $U_{L_\square} (\varepsilon_K e_{L_\square})$, which is again mapped onto $\mathfrak{p}_{L_\square}^{\varepsilon_K e_{L_\square}} = p_K^{\varepsilon_K} (\mathscr{O}_{L_\square})_+$ under the $p_K$-adic logarithm
    \begin{equation*}
U_{L_\square} (1) \otimes_{\ZZ_{p_K}} \QQ_{p_K}
        \xrightarrow{\cong, \log}
            (L_{\square})_+.
    \end{equation*}
    Hence we have the commutative diagram
    \begin{equation*}
        \begin{tikzcd}[column sep=3.5em]
            {G_{L_1}^\ab(\varepsilon_K e_{L_1})} & {U_{L_1}(\varepsilon_K e_{L_1})} & {p_K^{\varepsilon_K} (\mathscr{O}_{L_1})_+} \\
            {G_{L_2}^\ab(\varepsilon_K e_{L_2})} & {U_{L_2}(\varepsilon_K e_{L_2})} & {p_K^{\varepsilon_K} (\mathscr{O}_{L_2})_+}
            \arrow["{{{\cong,\Art_{L_1}^{-1}}}}", from=1-1, to=1-2]
            \arrow["{{{{{V'_{1,2}}}}}}", from=1-1, to=2-1]
            \arrow["{{{{{\cong, \log}}}}}", from=1-2, to=1-3]
            \arrow["\subseteq", from=1-2, to=2-2]
            \arrow["\subseteq", from=1-3, to=2-3]
            \arrow["{{{\cong,\Art_{L_2}^{-1}}}}", from=2-1, to=2-2]
            \arrow["{{{{{\cong,\log}}}}}", from=2-2, to=2-3]
        \end{tikzcd}
    \end{equation*}
    of $G_K^{m + 1}$-modules, which is compatible with the above commmutative diagram \eqref{equation:2025.10.15.11.35.10}.
    By passage to the limit, we see that the isomorphism of \cref{proposition:2025.10.15.11.35.03} restricts to the isomorphism
    \begin{equation*}
p_K^{\varepsilon_K} \cdot (\mathscr{O}_{k^m})_+ (G_K^{m + 1})
            =
                \varinjlim_{L / K} G_{L}^\ab (\varepsilon_K e_{L}) \,\,\,
        \xrightarrow{\cong} \,\,\,
p_K^{\varepsilon_K} (\mathscr{O}_{K^m})_+
            =
                \varinjlim_{L / K} p_K^{\varepsilon_K} (\mathscr{O}_L)_+,
    \end{equation*}
    where $L / K$ runs through the finite Galois subextensions of $K^{m} / K$.
    Therefore, we obtain the desired isomorphism, by multiplying both sides by $p_K^{-\varepsilon_K}$.
\end{proof}

\section{Abelian $p$-adic representations}
\label{section:2025.10.15.11.35.27}

Let $\ell$ be a prime number, $(\rho, V)$ an \emph{$\ell$-adic representation} of a profinite group $G$, i.e., a finite-dimensional $\QQ_\ell$-vector space $V$ equipped with a \emph{continuous} group homomorphism $\rho \colon G \to \Aut_{\QQ_\ell} (V)$.
We shall write $\rho_s$ for $\rho (s)$.
For an $\ell$-adic character $\chi \colon G \to \ZZ_\ell^\times$, we shall write $(\rho (\chi), V (\chi) \coloneqq V)$ for the $\ell$-adic representation of $G$ on which the $G$-action is defined by
\begin{equation*}
    {\rho (\chi)}_\sigma (v)
    \coloneqq
        \chi (\sigma) \cdot \rho_\sigma (v),
\end{equation*}
for $\sigma \in G$ and $v \in V (\chi)$.

\subsection*{Hodge-Tate numbers}

Let $K$ be a mixed-characteristic local field.
Recall that, for a $p_K$-adic representation $(\rho, V)$ of $G_K$ and an integer $i$,
the $i$\textsuperscript{th} \emph{Hodge-Tate number} $d^i_{\HT, K} (\rho, V)$ of $(\rho, V)$ is the dimension of the $K$-vector space
\begin{equation*}
    \left(
        \CC_{p_K} \otimes_{\QQ_{p_K}} V (-i)
    \right)^{G_K},
\end{equation*}
where $V (-i) = (\rho (-i), V (-i))$ denotes the $(-i)$\textsuperscript{th} Tate twist $V (\chi_K^{-i})$ of $V$.
From the theory of $p$-adic representations, it is known that
\begin{equation*}
\sum_{i \in \ZZ} d^i_{\HT, K} (\rho, V)
    \leq
        \dim_{\QQ_{p_K}} (V),
\end{equation*}
and we say that $(\rho, V)$ is \emph{Hodge-Tate} when the equality holds (cf. \cite[\S5.1]{FontaineOuyang}).

\begin{lemma}
    \label{lemma:2025.10.15.11.35.34}
    Let $(\rho, V)$ be a $p_K$-adic representation of $G_K$.
    Then we have
    \begin{equation}
        \label{equation:2025.10.15.11.35.39}
\left(
                \CC_{p_K} \otimes_{\QQ_{p_K}} V
            \right)^{G_K}
        =
            \left(
                \mathscr{C} (({K^\alg})^{\Ker(\rho)}) \otimes_{\QQ_{p_K}} V
            \right)^{G_K}.
    \end{equation}
\end{lemma}

\begin{proof}
    We choose a basis $v_1, \dots, v_n$ of $V$.
    For each $\sigma \in G_K$, we shall write $(a_{ij} (\sigma)) \in \GL_n (\QQ_{p_K})$ for the matrix of the linear transformation $\rho_\sigma$ with respect to the basis $v_1, \dots, v_n$, so that
    \begin{equation*}
(\rho_\sigma (v_1) \, \cdots \, \rho_\sigma (v_n))
        =
            (v_1 \, \cdots \, v_n) (a_{ij} (\sigma)).
    \end{equation*}

    Suppose that $c_1, \dots, c_n \in \CC_{p_K}$, and that $c_1 \otimes v_1 + \dots + c_n \otimes v_n$ belongs to the left hand side of \eqref{equation:2025.10.15.11.35.39}.
    Then for all $\sigma \in G_K$,
    \begin{align*}
c_1 \otimes v_1 + \dots + c_n \otimes v_n
        &=
            \sigma (c_1) \otimes \rho_\sigma (v_1)
            +
            \dots
            +
            \sigma (c_n) \otimes \rho_\sigma (v_n) \\
        &=
            \left(
                \sum_{j = 1}^n \sigma(c_j) a_{1j} (\sigma)
            \right) \otimes v_1
            +
            \dots
            +
            \left(
                \sum_{j = 1}^n \sigma(c_j) a_{nj} (\sigma)
            \right) \otimes v_n,
    \end{align*}
    and hence
    \begin{equation*}
\begin{pmatrix}
                c_1 \\
                \vdots \\
                c_n
            \end{pmatrix}
        =
            (a_{ij} (\sigma))
            \begin{pmatrix}
                \sigma (c_1) \\
                \vdots \\
                \sigma (c_n)
            \end{pmatrix}.
    \end{equation*}
    In particular, we have
    \begin{equation*}
c_1 \otimes v_1 + \dots + c_n \otimes v_n
        \in
\CC_{p_K}^{\Ker (\rho)}
                \otimes_{\QQ_{p_K}}
                V
            =
                \mathscr{C} ({(K^\alg)^{\Ker(\rho)}})
                \otimes_{\QQ_{p_K}}
                V,
    \end{equation*}
    since it holds that $\sigma (c_1) = c_1, \dots, \sigma (c_n) = c_n$ for all $\sigma \in \Ker (\rho)$.
    (Note that, for any closed subgroup $H$ of $G_K$,
    \begin{equation*}
\CC_{p_K}^{H}
        =
            \mathscr{C} ((K^\alg)^H)
    \end{equation*}
    by the theorem of Ax-Sen-Tate---cf. \cites{Tate1967}{Ax1970}[Chap. 3]{FontaineOuyang}.)
\end{proof}

\begin{definition}
    \label{definition:2025.10.15.11.35.50}
    Let $G$ be a profinite group, and let $(\rho, V)$ be an $\ell$-adic representation of $G$ for a prime number $\ell$.
    We shall say that $(\rho, V)$ is an \emph{$m$-step solvable} $\ell$-adic representation of $G$ for an integer $m \geq 0$ if $\rho$ annihilates $G^{[m]}$.
    We shall also say that $(\rho, V)$ is \emph{abelian} if it is $1$-step solvable.
\end{definition}

Let $G$ be a filtered profinite group of $\MLF^{m + 1}$-type for an integer $m \geq 1$; we set $p \coloneqq p (G^\ab)$.
Let $(\rho, V)$ and $\chi$ be a $p$-adic representation and a $p$-adic character of $G$, respectively.
We shall write
\begin{equation*}
d_{\HT, m}^i (G, \chi, \rho, V)
    \coloneqq
        \dim_{\QQ_p} \left(
            (\mathscr{C}_{k^m})_+ (G)
            \otimes_{\QQ_p}
            V (\chi^{-i})
        \right)^G / d(G^\ab)
\end{equation*}
for each $i \in \ZZ$.

\begin{proposition}
    \label{proposition:2025.10.15.11.36.01}
    Let $K$ be a mixed-characteristic local field, and let $m$ be an integer $\geq 1$.
    Given an \emph{$m$-step solvable} $p_K$-adic representation $(\rho, V)$ of $G_K$, it holds that
    \begin{equation*}
d^i_{\HT, m} (G_K^{m + 1}, \chi (G_K^2), \rho, V)
        =
            d^i_{\HT, K} (\rho, V)
    \end{equation*}
    for each $i \in \ZZ$.
\end{proposition}

Intuitively speaking, the $i$\textsuperscript{th} \emph{Hodge-Tate number} of $(\rho, V)$ (and hence the issue of \emph{whether or not $(\rho, V)$ is Hodge-Tate}) can be determined group-theoretically from the \emph{filtered} profinite group $G_K^{m + 1}$ and its action on $V$, if $(\rho, V)$ is an \emph{$m$-step solvable} representation of $G_K$.

\begin{proof}
    We have
    \begin{align*}
d^i_{\HT, m} (G_K^{m + 1}, \chi (G_K^2), \rho, V)
        &=
            \dim_{\QQ_{p_K}} \left(
                \mathscr{C} (K^m) \otimes_{\QQ_{p_K}} V (-i)
            \right)^{G_K} / d_K \\
        &=
            \dim_{K} \left(
                \mathscr{C} (K^m) \otimes_{\QQ_{p_K}} V (-i)
            \right)^{G_K}
    \end{align*}
    from \cref{proposition:2025.10.15.11.34.21,proposition:2025.10.15.11.34.35,proposition:2025.10.15.11.38.37}.
    Since $\rho$ and $\chi_K$ annihilates $G_K^{[m]}$, we see that
    \begin{equation*}
\left(
\CC_{p_K}
                \otimes_{\QQ_{p_K}}
                    V (-i)
            \right)^{G_K}
        \supseteq
            \left(
\mathscr{C} (K^m)
                \otimes_{\QQ_{p_K}}
                    V (-i)
            \right)^{G_K}
        \supseteq
            \left(
\mathscr{C} (({K}^{\alg})^{\Ker (\rho (-i))})
                \otimes_{\QQ_{p_K}}
                    V (-i)
            \right)^{G_K},
    \end{equation*}
    and deduce the desired equality from \cref{lemma:2025.10.15.11.35.34}.
\end{proof}

\subsection*{Uniformizing representations}

Let $(\rho, V)$ be a $p_K$-adic representation of $G_K$ for a mixed-characteristic local field $K$, and $E / K$ a finite extension such that $E / \QQ_{p_K}$ is Galois.
Suppose that $V$ is an $E$-vector space of dimension $1$ and the $G_K$-action on $V$ is $E$-linear.
Then we see that $\rho \colon G_K \to \Aut_{\QQ_{p_K}} (V)$ factors through $\rho \colon G_K^\ab \to E^\times$.
In particular, $(\rho, V)$ is an \emph{abelian} representation.
We say that a representation $(\rho, V)$ of this type is \emph{uniformizing} if there exist an open subgroup $I \subseteq U_K$ and a field homomorphism $\iota \colon K \to E$ such that
\begin{equation*}
(\rho \circ \Art_K) \vert_I
    =
        \iota^\times \vert_I,
\end{equation*}
where $\iota^\times \colon K^\times \to E^\times$ is the group homomorphism induced by $\iota$.

\begin{example}
    \label{example:2025.10.15.11.36.19}
    Given any finite extension $E / K$ such that $E / \QQ_{p_K}$ is Galois,
    $V = E_+$ can be regarded as a uniformizing representation by local class field theory.
    More precisely, we define the $G_K$-action on $V$ via the below composition
    \begin{equation*}
        \rho
        \colon \,\,\,
            G_K^\ab \,\,\,
        \xrightarrow{\Ver} \,\,\,
            G_E^\ab \, (
                \cong
                    G_E^\ab (0) \oplus \Gal (E^\nr / E)
            ) \,\,\,
        \twoheadrightarrow \,\,\,
            G_E^\ab (0) \,\,\,
        \to \,\,\,
            U_E
    \end{equation*}
    of continuous homomorphisms, where the second (resp. third) arrow is the natural surjection (resp. the isomorphism $\Art_E^{-1}$).
\end{example}

\begin{proposition}
    \label{proposition:2025.10.15.11.36.28}
    Let $K$ be a mixed-characteristic local field, and let $E / K$ be a finite extension such that $E / \QQ_{p_K}$ is Galois.
    Suppose that $(\rho, V)$ is an $E$-linear representation of $G_K$, of $E$-dimension $1$.
    Then $(\rho, V)$ is a uniformizing representation if and only if
    \begin{equation*}
d_{\HT, K}^i (\rho, V)
        =
            \begin{cases*}
                [E : K] ([K : \QQ_{p_K}] - 1) & $i = 0$ \\
                [E : K] & $i = 1$ \\
            \end{cases*}.
    \end{equation*}
\end{proposition}

\begin{proof}
    cf. \cites[Chap. III, A5]{Serre1968}[\S3]{Mochizuki1997}.
\end{proof}

\begin{corollary}
    \label{corollary:2025.10.15.11.36.37}
    Let $K_\circ$ and $K_\bullet$ be mixed-characteristic local fields, $\alpha_2 \colon G_{K_\circ}^2 \to G_{K_\bullet}^2$ an isomorphism of \emph{filtered} profinite groups.
    Suppose that $E / \QQ_{p_{K_\circ}} (= \QQ_{p_{K_\bullet}})$ is a finite Galois extension containing both $K_\circ$ and $K_\bullet$,
    and that $(\rho, V)$ is an $E$-linear representation of $G_{K_\circ}$, of $E$-dimension $1$.
    Then $(\rho \circ \alpha_1^{-1}, V)$ is a uniformizing representation of $G_{K_\bullet}$ if and only if $(\rho, V)$ is a uniformizing representation of $G_{K_\circ}$,
    where $\alpha_1 \colon G_{K_\circ}^\ab \to G_{K_\bullet}^\ab$ is the isomorphism induced by $\alpha_2$.
\end{corollary}

\begin{proof}
    It follows immediately from \cref{proposition:2025.10.15.11.36.28,proposition:2025.10.15.11.36.01}.
\end{proof}
 
\section{Proofs of the main theorems}
\label{section:2025.10.15.11.36.45}

\begin{lemma}[{\cite[Lemma 4.1]{Mochizuki1997}}]
    \label{lemma:2025.10.15.11.36.50}
    Let $K$ be a mixed-characteristic local field, and $I$ an open subgroup of $U_K$.
    Then the sub-$\QQ_{p_K}$-vector space generated by $I$ in $K$ equals $K$.
\end{lemma}

\begin{proof}
    We denote by $W$ the sub-$\QQ_{p_K}$-vector space generated by $I$ in $K$.
    First, observe that $I$ is an open subset of $K$, since $U_K$ is open in $K$.
    Then note that, for each $w \in W$, $w + I \subseteq W$; hence $W$ is also an open subgroup of $K$.
    Therefore, the $\QQ_{p_K}$-vector space $K/W$ ($\cong \QQ_{p_K}^{\oplus d}$, where $d = d_K - \dim_{\QQ_{p_K}} (W)$) is discrete, meaning that $d_K = \dim_{\QQ_{p_K}} (W)$.
\end{proof}

\begin{theorem}
    \label{theorem:2025.10.22.02.13.03}
    Assume that there exists an isomorphism
    \begin{equation*}
        \alpha_2
        \colon
            G_{K_\circ}^2
        \overset{\cong}{\to}
            G_{K_\bullet}^2
    \end{equation*}
    of \emph{filtered} profinite groups.
    Then there exists an isomorphism $f \colon K_\circ \to K_\bullet$.
\end{theorem}

\begin{proof}
    We denote by $\alpha_1$ the isomorphism $G_{K_\circ}^\ab \to G_{K_\bullet}^\ab$ induced by $\alpha_2$.
    We set $p \coloneqq p_{K_\circ} = p_{K_\bullet}$, and choose a finite Galois extension $E / \QQ_{p}$ that contains both $K_\circ$ and $K_\bullet$.
    As we have seen in \cref{example:2025.10.15.11.36.19}, we have the natural uniformizing representation $(\rho_\circ, V \coloneqq E_+)$ of $G_{K_\circ}$; it is clear from \cref{corollary:2025.10.15.11.36.37} that $(\rho_\bullet \coloneqq \rho_\circ \circ \alpha_{1}^{-1}, V)$ is also a uniformizing representation of $G_{K_\bullet}$.
    Hence there exist an open subgroup $I_\circ$ (resp. $I_\bullet$) of $U_{K_\circ}$ (resp. $U_{K_\bullet}$) and a field homomorphism $\iota_\circ \colon K_\circ \to E$ (resp. $\iota_\bullet \colon K_\bullet \to E$) such that $\alpha_1 (I_\circ) = I_\bullet$ and the diagram
    \begin{equation*}
        \begin{tikzcd}
            & {K_\circ^\times} \\
            {I_\circ} & {U_{K_\circ}} & {G_{K_\circ}^\ab} & {E^\times} \\
            {I_\bullet} & {U_{K_\bullet}} & {G_{K_\bullet}^\ab} & {E^\times} \\
            & {K_\bullet^\times}
            \arrow["{{\cong,{{\alpha_1}}}}", from=2-3, to=3-3]
            \arrow[Rightarrow, no head, from=2-4, to=3-4]
            \arrow["{{\cong,{{\alpha_1\vert_{I_\circ}}}}}", from=2-1, to=3-1]
            \arrow["{\rho_\circ}", from=2-3, to=2-4]
            \arrow["{{{{\rho_\bullet}}}}", from=3-3, to=3-4]
            \arrow["{{{{\iota_\circ^\times}}}}"{description}, curve={height=-12pt}, from=1-2, to=2-4]
            \arrow["{{{{\iota_\bullet^\times}}}}"{description}, curve={height=12pt}, from=4-2, to=3-4]
            \arrow["\subseteq"{description}, from=2-1, to=1-2]
            \arrow["\subseteq"{description}, from=3-1, to=4-2]
            \arrow["{{\cong,{{\alpha_1\vert_{U_{K_\circ}}}}}}", from=2-2, to=3-2]
            \arrow["{{{{\Art_{K_\bullet}}}}}", from=3-2, to=3-3]
            \arrow["\subseteq", from=2-1, to=2-2]
            \arrow["{{{{\Art_{K_\circ}}}}}", from=2-2, to=2-3]
            \arrow["\subseteq", from=3-1, to=3-2]
        \end{tikzcd}
    \end{equation*}
    commutes.
    In particular, $\iota_\bullet \vert_{I_\bullet} \circ \alpha_1 \vert_{I_\circ} = \iota_\circ \vert_{I_\circ}$, and by \cref{lemma:2025.10.15.11.36.50}, $\iota_\circ (K_\circ) = \iota_\bullet (K_\bullet)$ in $E$.
    Therefore, we have the following field isomorphism:
    \begin{equation*}
        f
        \colon
            K_\circ
        \xrightarrow{\cong, \iota_\circ}
\iota_\circ (K_\circ)
            =
                \iota_\bullet (K_\bullet)
        \xrightarrow{\cong, \iota_\bullet^{-1}}
            K_\bullet.
    \end{equation*}
\end{proof}

\begin{theorem}
    \label{theorem:2025.10.22.02.13.10}
    Let $m$ be an integer $\geq 0$.
    For an isomorphism
    \begin{equation*}
        \alpha_{m + 3}
        \colon
            G_{K_\circ}^{m + 3}
        \xrightarrow{\cong}
            G_{K_\bullet}^{m + 3}
    \end{equation*}
    of \emph{filtered} profinite groups, there exists an isomorphism
    $\theta_{m + 1} \colon K_{\circ}^{m + 1} \to K_{\bullet}^{m + 1}$ such that
    \begin{equation*}
\alpha_{m + 1} (\sigma)
        =
\theta_{m + 1}
            \circ
                \sigma
            \circ
                \theta_{m + 1}^{-1}
    \end{equation*}
    for every $\sigma \in G_{K_\circ}^{m + 1}$, where $\alpha_{m + 1} \colon G_{K_\circ}^{m + 1} \to G_{K_\bullet}^{m + 1}$ is the isomorphism induced by $\alpha_{m + 3}$.
    Moreover,
    \begin{enumerate}[label=(\roman*)]
        \item if $m \geq 1$, the isomorphism $\theta_{m + 1}$ is uniquely determined by $\alpha_{m + 3}$;
        \item if $m = 0$, the isomorphism $\theta_{m + 1} \vert_{K_\circ} \colon K_\circ \to K_\bullet$ is uniquely determined by $\alpha_{m + 3}$.
    \end{enumerate}
\end{theorem}

\begin{proof}
    We keep the notation and hypotheses of the proof of \cref{theorem:2025.10.22.02.13.03}.

    \vspace{1em}
    \noindent \emph{Uniqueness}.
    Suppose that both isomorphisms $\theta_{m + 1, 1}, \theta_{m + 1, 2} \colon K_\circ^{m + 1} \to K_\bullet^{m + 1}$ satisfy the condition.
    Then there exist isomorphisms $\theta_1, \theta_2 \colon K_\circ^\alg \to K_\bullet^\alg$ that respectively extend $\theta_{m + 1, 1}, \theta_{m + 1, 2}$; we have
    \begin{equation*}
\gamma
            \coloneqq
                (\theta_2)^{-1} \circ \theta_1
        \in
            \Gal (K_\circ^\alg / \QQ_p)
    \end{equation*}
    and
    \begin{equation}
        \label{equation:2025.10.15.11.37.02}
\gamma \vert_{K_\circ^{m + 1}}
            \circ
                \sigma
            \circ
                \gamma^{-1} \vert_{K_\circ^{m + 1}}
        =
(\theta_{m + 1, 2})^{-1}
            \circ
                \theta_{m + 1, 1}
            \circ
                \sigma
            \circ
                (\theta_{m + 1, 1})^{-1}
            \circ
                \theta_{m + 1, 2}
        =
            \sigma, \quad
    \end{equation}
    for all $\sigma \in G_{K_\circ}^{m + 1}$.
    Hence we see that, for all $x \in K_\circ^\times$,
    \begin{equation*}
\gamma \vert_{K_\circ^\ab}
            \circ
                \Art_{K_\circ} (x)
            \circ
                \gamma^{-1} \vert_{K_\circ^\ab}
        =
            \Art_{K_\circ} (x),
    \end{equation*}
    and $\gamma (x) = x$ by local class field theory; it follows that $\gamma \in \Gal (K_\circ^\alg / K_\circ)$ (i.e., $\theta_{m + 1, 1} \vert_{K_\circ} = \theta_{m + 1, 2} \vert_{K_\circ}$).
    Furthermore, $\gamma \vert_{K_\circ^{m + 1}} \in Z (G_{K_\circ}^{m + 1})$ by \eqref{equation:2025.10.15.11.37.02}; together with the fact that $Z (G_{K_\circ}^{m + 1})$ is trivial for $m \geq 1$ (cf. \cref{proposition:2025.10.15.11.37.22}), we conclude that $\gamma \vert_{K_\circ^{m + 1}} = 1$ (i.e., $\theta_{m + 1, 1} = \theta_{m + 1, 2}$) if $m \geq 1$.

    \vspace{1em}
    \noindent \emph{Existence}.
    We suppose that, for each $i \in \left\{ 1, 2 \right\}$,
    \begin{itemize}
        \item $L_{i, \square}$ is a finite Galois extension of $K_\square$ contained in $K_\square^{m + 1}$ for each $\square \in \left\{ \circ, \bullet \right\}$;
        \item $H_{i, \square} = \Gal (K_\square^{m + 3} / L_{i, \square}) \,(\supseteq (G_{K_\square}^{m + 3})^{[m + 1]})$ for each $\square \in \left\{ \circ, \bullet \right\}$;
        \item $H_{i, \bullet} = \alpha_{m + 3} (H_{i, \circ})$,
    \end{itemize}
    and that $L_{1, \circ} \subseteq L_{2, \circ}$.
    Then $\alpha_{m + 3} \vert_{H_{i, \circ}} \colon H_{i, \circ} \to H_{i, \bullet}$ is an isomorphism of \emph{filtered} profinite groups by \cref{lemma:2025.10.15.11.33.03}.
    Hence $\alpha_{m + 3} \vert_{H_{i, \circ}}$ induces an isomorphism $\alpha_{2, i} \colon G_{L_{i, \circ}}^2 = H_{i, \circ}^2 \to G_{L_{i, \bullet}}^2 = H_{i, \bullet}^2$ of filtered profinite groups.
    As we have seen in the proof of \cref{theorem:2025.10.15.11.32.12}, there exist an open subgroup $I_{i, \circ}$ (resp. $I_{i, \bullet}$) of $U_{L_{i, \circ}}$ (resp. $U_{L_{i, \bullet}}$) and a field isomorphism $\theta_{L_{i, \circ}} \colon L_{i, \circ} \to L_{i, \bullet}$ such that $\alpha_{1,i} (I_{i, \circ}) = I_{i, \bullet}$ and $\theta_{L_{i, \circ}}$ (set-theoretically) extends the group isomorphism $\alpha_{1, i} \vert_{I_{i,\circ}} \colon I_{i, \circ} \to I_{i, \bullet}$, where $\alpha_{1, i} \colon G_{L_{i, \circ}}^\ab \to G_{L_{i, \bullet}}^\ab$ is the isomorphism induced by $\alpha_{2, i}$.
    We can assume without loss of generality that $I_{1, \circ} \subseteq I_{2, \circ}$, replacing $I_{1, \circ}$ with $I_{1, \circ} \cap I_{2, \circ}$ if necessary; the diagram
    \begin{equation*}
        \begin{tikzcd}
            {I_{1,\circ}} & {I_{2,\circ}} \\
            {I_{1,\bullet}} & {I_{2,\bullet}}
            \arrow["\subseteq", from=1-1, to=1-2]
            \arrow["{\alpha_{1,1}\vert_{I_{1,\circ}}}", from=1-1, to=2-1]
            \arrow["{\alpha_{1,2}\vert_{I_{2,\circ}}}", from=1-2, to=2-2]
            \arrow["\subseteq", from=2-1, to=2-2]
        \end{tikzcd}
    \end{equation*}
    commutes by definition.
    It follows immediately from \cref{lemma:2025.10.15.11.36.50} that $\theta_{L_{2, \circ}}$ restricts to $\theta_{L_{1, \circ}}$;
    by passage to the limit, we obtain
    \begin{equation*}
        \theta_{m + 1}
        \colon
            K_\circ^{m + 1}
        \to
            K_\bullet^{m + 1}.
    \end{equation*}

    It remains to check that $\theta_{m + 1}$ satisfies the stated condition: since $L_{i, \circ}$ is an arbitrary finite Galois extension of $K_\circ$ contained in $K_\circ^{m + 1}$, it suffices to show that, for all $\gamma_\circ \in G_{K_\circ}^{m + 3}$ and $x \in L_{i, \circ}$,
    \begin{equation}
        \label{equation:2025.10.22.02.23.46}
\theta_{L_{i, \circ}} (\gamma_\circ (x))
        =
            \gamma_\bullet (\theta_{L_{i, \circ}} (x)),
    \end{equation}
    where $\gamma_\bullet = \alpha_{m + 3} (\gamma_\circ)$.
    By \cref{lemma:2025.10.15.11.36.50}, it is reduced to showing that \eqref{equation:2025.10.22.02.23.46} holds for all $x \in I_{i, \circ}$, that is,
    \begin{equation*}
\alpha_{1, i} (\gamma_\circ (x))
        =
            \gamma_\bullet (\alpha_{1, i} (x))
    \end{equation*}
    for all $x \in I_{i, \circ}$.
    This holds since we are regarding $I_{i, \square}$ as a subgroup of $G_{L_{i, \square}}^\ab$ via $\Art_{L_{i, \square}}$, and
    \begin{equation*}
\alpha_{1, i} (
\gamma_\circ \vert_{L_{i, \circ}^\ab}
                \circ
                    \sigma
                \circ
                    \gamma_\circ^{-1} \vert_{L_{i, \circ}^\ab}
            )
        =
\gamma_\bullet \vert_{L_{i, \bullet}^\ab}
            \circ
                \alpha_{1, i}(\sigma)
            \circ
                \gamma_\bullet^{-1} \vert_{L_{i, \bullet}^\ab}
    \end{equation*}
    for all $\sigma \in G_{L_{i, \circ}}^\ab$.
\end{proof}

\appendix

\section{Center-freeness of $G_K^{m}$, $m \geq 2$}
\label{section:2025.10.15.11.37.10}

This appendix is devoted to the proof of the following proposition.

\begin{proposition}
    \label{proposition:2025.10.15.11.37.22}
    Let $K$ be a mixed-characteristic local field. Then
    \begin{equation*}
Z (G_K^{m + 1})
        =
            \left\{ 1 \right\}
    \end{equation*}
    for all integer $m \geq 1$.
\end{proposition}

\begin{remark}
    \Cref{proposition:2025.10.15.11.37.22} has been originally proved by S. Ladkani \cite{Ladkani1996} for the case $m = 1$.
    (It is known that the assertion for general $m \geq 1$ follows from the case $m = 1$ by induction, cf. \cite[Proof of Proposition 1.1 (ix)]{SaidiTamagawa2022}.)
    In this appendix, we provide an alternative proof of the proposition.
\end{remark}
To prove \cref{proposition:2025.10.15.11.37.22}, we first give a proof of a weaker statement.

\begin{lemma}
    \label{lemma:2025.10.15.11.37.28}
    Let $K$ be a mixed-characteristic local field.
    Then
    \begin{equation*}
Z (G_K^{m + 1})
        \subseteq
            \Gal (K^{m + 1} / K^m)
    \end{equation*}
    for all integer $m \geq 0$.
\end{lemma}

\begin{proof}
    Suppose that $\gamma \in Z (G_K^{m + 1})$.
    Then $\gamma \circ \sigma \circ \gamma^{-1} = \sigma$ for all $\sigma \in G_{K}^{m + 1}$.
    Let $L / K$ be a finite Galois subextension of $K^m / K$, so that $L^\ab \subseteq K^{m + 1}$.
    We see that, for all $x \in L^\times$,
    \begin{equation*}
\gamma \vert_{L^\ab}
            \circ
                \Art_{L}(x)
            \circ
                \gamma^{-1} \vert_{L^\ab}
        =
            \Art_{L} (x),
    \end{equation*}
    and $\gamma (x) = x$ by local class field theory.
    Therefore, $\gamma \in \Gal (K^{m + 1} / L)$, and the assertion follows as the subextension $L / K$ is arbitrary.
\end{proof}

\begin{remark}
    For the case in which the base field is \emph{torally Kummer-faithful} (see, e.g., \cite[Definition 1.5]{Mochizuki2015}, for the definition of (torally) Kummer-faithful fields), a claim similar to that of \cref{lemma:2025.10.15.11.37.28} holds: let $k$ be a torally Kummer-faithful field, and $m$ an integer $\geq 0$.
    We shall write $\Primes_{\times / k}$ for the \emph{set of prime numbers invertible in $k$}, $\widehat{\ZZ}_{\times / k}$ for the \emph{maximal pro-$\Primes_{\times / k}$ quotient of $\widehat{\ZZ}$} and
    \begin{equation*}
        \chi_{\cycl, k}
        \colon
            G_k
        \to
            (
                \widehat{\ZZ}_{\times / k}
            )^\times
    \end{equation*}
    for the ($\Primes_{\times / k}$-adic) \emph{cyclotomic character} of $k$.
    Then
    \begin{equation*}
Z (G_k^{m + 1}) \cap \Ker (
                \chi_{\cycl, k}
                \colon
                    G_k^{m + 1}
                \to
                    (\widehat{\ZZ}_{\times / k})^\times
            )
        \subseteq
            \Gal(k^{m + 1} / k^m)
    \end{equation*}
    holds.
    (Compare \cite[Proposition 1.5]{Hoshi2017}.)

    There is nothing to show if $m = 0$. For the case $m \geq 1$, we give a proof by contradiction.
    Assume that there exists an element $\gamma \in Z (G_k^{m + 1}) \cap \Ker (\chi_{\cycl, k})$ which does not belong to $\Gal (k^{m + 1} / k^m)$;
    let $\tilde \gamma \in G_k$ be a lifting of $\gamma$.
    Then we can choose a finite Galois subextension $l / k$ of $k^m / k$, such that the corresponding open normal subgroup $\Gal (k^{m + 1} / l)$ does not contain $\gamma$.

    Since $k$ is torally Kummer-faithful, we have the \emph{injective} homomorphism
    \begin{equation*}
l^\times
        \to
            \varprojlim_n l^\times / (l^\times)^n
        \cong
            H^1 (G_l, \varprojlim_n \mu_n (k^\sep))
    \end{equation*}
    of $G_k$-modules.
    ($n$ runs through the integers $\geq 1$ whose prime factors belong to $\Primes_{\times / k}$.)
    We deduce a contradiction by claiming that $\tilde{\gamma}$ acts trivially on $H^1 (G_l, \varprojlim \mu_n (k^\sep))$, and hence on $l^\times$.

    For all $\sigma \in G_l$, we have
    \begin{equation*}
\xi_\sigma
            \coloneqq
                \sigma^{-1} \tilde{\gamma}^{-1} \sigma \tilde{\gamma}
        \in
            G_k^{[m + 1]} \,
            (
                \subseteq
                    G_k^{[m]}
                \subseteq
                    G_l
            ),
    \end{equation*}
    since $\gamma \in Z (G_k^{m + 1})$.
    On the other hand, the action of $\tilde{\gamma}$ on $H^1 (G_l, \mu_n (k^\sep))$ for each $n$ is determined as follows: for each $1$-cocycle (i.e., crossed homomorphism) $\omega \colon G_l \to \mu_n (k^\sep)$,
    \begin{equation*}
\tilde{\gamma} \omega (-)
        =
            \tilde{\gamma} \cdot \omega
            (\tilde{\gamma}^{-1} \cdot - \cdot \tilde{\gamma})
        =
            \omega (\tilde{\gamma}^{-1} \cdot - \cdot \tilde{\gamma}).
    \end{equation*}
    Therefore, it suffices to show that
    \begin{equation*}
\tilde{\gamma} \omega (-) / \omega (-)
        =
            \omega (\tilde{\gamma}^{-1} \cdot - \cdot \tilde{\gamma})
            /
            \omega (-)
        =
            \omega(- \cdot \xi_{(-)})
            /
            \omega (-)
        =
            (-) \cdot \omega (\xi_{(-)})
    \end{equation*}
    is a $1$-coboundary.
    This is straightforward, since it holds that $\omega (\xi) = 1$ for all $1$-cocycle $\omega$ and $\xi \in G_k^{[m + 1]}$, which follows from the fact that $\omega \vert_{G_k^{[m]}}$ is a group homomorphism (as $G_k^{[m]}$ acts trivially on $\mu_n(k^\sep)$).
\end{remark}

\begin{lemma}
    \label{lemma:2025.10.15.11.37.39}
    Let $K$ be a mixed-characteristic local field, and let $L / K$ be a finite extension with inclusion map $\iota \colon K \to L$.
    \begin{enumerate}
        \item The group homomorphism $\widehat{\iota^\times} \colon \widehat{K^\times} \to \widehat{L^\times}$ is injective.
        \item The transfer map $\Ver \colon G_K^\ab \to G_L^\ab$ is injective.
        \item It holds that $(\widehat{L^\times})^{G_K} = \widehat{K^\times}$.
    \end{enumerate}
\end{lemma}

\begin{proof}
    (1) We set $e \coloneqq e_L / e_K$.
    We consider the following commutative diagram (of abelian groups) with exact rows:
    \begin{equation}
        \label{equation:2025.10.15.11.37.46}
        \begin{tikzcd}
            1 & {U_K} & {K^\times} & {\ZZ_+} & 1 \\
            1 & {U_L} & {L^\times} & {\ZZ_+} & 1
            \arrow[from=1-1, to=1-2]
            \arrow[from=1-2, to=1-3]
            \arrow[from=1-2, to=2-2]
            \arrow["{{\ord_K}}", from=1-3, to=1-4]
            \arrow["{\iota^\times}", from=1-3, to=2-3]
            \arrow[from=1-4, to=1-5]
            \arrow["{{e\cdot (-)}}", from=1-4, to=2-4]
            \arrow[from=2-1, to=2-2]
            \arrow[from=2-2, to=2-3]
            \arrow["{{\ord_L}}", from=2-3, to=2-4]
            \arrow[from=2-4, to=2-5]
        \end{tikzcd}.
    \end{equation}
    By profinite completion, we obtain the following commutative diagram, in which all rows are exact:
    \begin{equation}
        \label{equation:2025.10.15.11.37.54}
        \begin{tikzcd}
            & 1 & {U_K} & {K^\times} & {\ZZ_+} & 1 \\
            & 1 & {U_L} & {L^\times} & {\ZZ_+} & 1 \\
            1 & {U_K} & {\widehat{K^\times}} & {\widehat{\ZZ}_+} & 1 \\
            1 & {U_L} & {\widehat{L^\times}} & {\widehat{\ZZ}_+} & 1
            \arrow[from=1-2, to=1-3]
            \arrow[from=1-3, to=1-4]
            \arrow[from=1-3, to=2-3]
            \arrow[dashed, from=1-3, to=3-2]
            \arrow[from=1-4, to=1-5]
            \arrow["{\iota^\times}", from=1-4, to=2-4]
            \arrow[dashed, from=1-4, to=3-3]
            \arrow[from=1-5, to=1-6]
            \arrow["{{e\cdot(-)}}", from=1-5, to=2-5]
            \arrow[dashed, from=1-5, to=3-4]
            \arrow[from=2-2, to=2-3]
            \arrow[from=2-3, to=2-4]
            \arrow[dashed, from=2-3, to=4-2]
            \arrow[from=2-4, to=2-5]
            \arrow[dashed, from=2-4, to=4-3]
            \arrow[from=2-5, to=2-6]
            \arrow[dashed, from=2-5, to=4-4]
            \arrow[from=3-1, to=3-2]
            \arrow[from=3-2, to=3-3]
            \arrow[from=3-2, to=4-2]
            \arrow[from=3-3, to=3-4]
            \arrow["{\widehat{\iota^\times}}", from=3-3, to=4-3]
            \arrow[from=3-4, to=3-5]
            \arrow["{{{e\cdot (-)}}}", from=3-4, to=4-4]
            \arrow[from=4-1, to=4-2]
            \arrow[from=4-2, to=4-3]
            \arrow[from=4-3, to=4-4]
            \arrow[from=4-4, to=4-5]
        \end{tikzcd}.
    \end{equation}
    Since $e$ is not a zero-divisor in $\widehat{\ZZ}$, we can conclude that the map $\widehat{\iota^\times} \colon \widehat{K^\times} \to \widehat{L^\times}$ is injective.

    (2) $\Art_K$ and $\Art_L$ respectively induce the isomorphisms $(\Art_K)^\wedge \colon \widehat{K^\times} \to G_K^\ab$ and $(\Art_L)^\wedge \colon \widehat{L^\times} \to G_L^\ab$ (cf. p.\pageref{2025.10.15.11.34.04});
    these fit into the following commutative diagram:
    \begin{equation*}
        \begin{tikzcd}
            {\widehat{K^\times}} & {G_K^\ab} \\
            {\widehat{L^\times}} & {G_L^\ab}
            \arrow["{\cong}", from=1-1, to=1-2]
            \arrow["{\widehat{\iota^\times}}", from=1-1, to=2-1]
            \arrow["\Ver", from=1-2, to=2-2]
            \arrow["{\cong}", from=2-1, to=2-2]
        \end{tikzcd}.
    \end{equation*}
    Hence the injectivity of $\Ver$ is implied by that of $\widehat{\iota^\times}$, which we have already seen in (1).

    (3) By (1), we can assume without loss of generality that $L / K$ is a finite Galois extension, with Galois group $G = \Gal (L / K)$.
    Then it follows that \eqref{equation:2025.10.15.11.37.46} and \eqref{equation:2025.10.15.11.37.54} are also commutative diagrams of $G$-modules;
    in \eqref{equation:2025.10.15.11.37.54}, we see that the second and fourth rows from the top induce the long exact sequences, and make the following diagram commutative:
    \begin{equation*}
        \begin{tikzcd}
            1 & {U_K} & {K^\times} & {\ZZ_+} & {H^1(G, U_L)} & {H^1(G,L^\times)} \\
            1 & {U_K} & {(\widehat{L^\times})^G} & {\widehat\ZZ_+} & {H^1(G,U_L)} & {H^1(G,\widehat{L^\times})}
            \arrow[from=1-1, to=1-2]
            \arrow[from=1-2, to=1-3]
            \arrow[Rightarrow, no head, from=1-2, to=2-2]
            \arrow["{{\ord_L\vert_{K^\times}}}", from=1-3, to=1-4]
            \arrow[dashed, from=1-3, to=2-3]
            \arrow["\delta", from=1-4, to=1-5]
            \arrow[dashed, from=1-4, to=2-4]
            \arrow[from=1-5, to=1-6]
            \arrow[Rightarrow, no head, from=1-5, to=2-5]
            \arrow[dashed, from=1-6, to=2-6]
            \arrow[from=2-1, to=2-2]
            \arrow[from=2-2, to=2-3]
            \arrow[from=2-3, to=2-4]
            \arrow["\delta", from=2-4, to=2-5]
            \arrow[from=2-5, to=2-6]
        \end{tikzcd}.
    \end{equation*}
    The connecting homomorphism $\delta \colon \ZZ_+ \to H^1 (G, U_L)$ induces the injective homomorphism
    \begin{equation*}
(
\Coker (\ord_L \vert_{K^\times})
                =
            ) \,
            (\ZZ / e \ZZ)_+ \,\,\,
        \to \,\,\,
            H^1 (G, U_L),
    \end{equation*}
    ---which is an isomorphism by Hilbert's Theorem 90---and as a result, $e \widehat{\ZZ}_+$ is annihilated by $\delta \colon \widehat{\ZZ}_+ \to H^1(G, U_L)$.
    Therefore, the diagram
    \begin{equation*}
        \begin{tikzcd}
            1 & {U_K} & {\widehat{K^\times}} & {\widehat\ZZ_+} & 1 \\
            1 & {U_K} & {(\widehat{L^\times})^{G}} & {\widehat\ZZ_+} & {H^1(G,U_L)}
            \arrow[from=1-1, to=1-2]
            \arrow[from=1-2, to=1-3]
            \arrow[Rightarrow, no head, from=1-2, to=2-2]
            \arrow[from=1-3, to=1-4]
            \arrow[from=1-3, to=2-3]
            \arrow[from=1-4, to=1-5]
            \arrow["{e\cdot (-)}", from=1-4, to=2-4]
            \arrow[from=1-5, to=2-5]
            \arrow[from=2-1, to=2-2]
            \arrow[from=2-2, to=2-3]
            \arrow[from=2-3, to=2-4]
            \arrow["\delta", from=2-4, to=2-5]
        \end{tikzcd}
    \end{equation*}
    with exact rows commutes, and yields the exact sequence of cokernels:
    \begin{equation*}
1
        \to
            (\widehat{L^\times})^{G} / \widehat{K^\times}
        \to
            (\widehat{\ZZ} / e \widehat{\ZZ})_+
        \xrightarrow{\cong}
            H^1 (G, U_L).
    \end{equation*}
    Hence $(\widehat{L^\times})^G = \widehat{K^\times}$.
\end{proof}

\begin{proof}[Proof of \cref{proposition:2025.10.15.11.37.22}]
    By definition, $Z (G_K^{m + 1})$ is precisely the set of $G_K$-invariant elements in $G_K^{m + 1}$, if we let $G_K$ act on $G_K^{m + 1}$ by conjugation.
    Hence it follows from \cref{lemma:2025.10.15.11.37.28} that
    \begin{equation*}
Z (G_K^{m + 1})
        =
            \Gal (K^{m + 1} / K^m)^{G_K}.
    \end{equation*}
    To demonstrate that $\Gal (K^{m + 1} / K^m)^{G_K}$ is trivial, we first note that $\Gal (K^{m + 1} / K^m)$ can be written as an inverse limit of profinite groups:
    \begin{multline*}
\Gal (K^{m + 1} / K^m)
        =
            G_K^{[m]} / G_K^{[m + 1]}
        =
            (G_K^{[m]})^\ab \\
        =
            \left(
                \bigcap_{H \in \mathscr{H}_m (G_K)} H
            \right)^\ab
        =
            \left(
                \varprojlim_{H \in \mathscr{H}_m (G_K)} H
            \right)^\ab
        =
            \varprojlim_{H \in \mathscr{H}_m (G_K)} H^\ab,
    \end{multline*}
    where $\mathscr{H}_m (G_K)$ is the set of open normal subgroups of $G_K$ containing $G_K^{[m]}$, ordered by reverse inclusion.
    Suppose that $L_1 / K$, $L_2 / K$ are finite Galois subextensions of $K^m / K$ with $L_1 \subseteq L_2$.
    It is clear from local class field theory that the diagram
    \begin{equation*}
        \begin{tikzcd}
            {G_{L_1}^\ab} & {L_1^\times} \\
            {G_{L_2}^\ab} & {L_2^\times}
            \arrow["{\Art_{L_1}}"', from=1-2, to=1-1]
            \arrow[from=2-1, to=1-1]
            \arrow["{\N_{L_2/L_1}}"', from=2-2, to=1-2]
            \arrow["{\Art_{L_2}}"', from=2-2, to=2-1]
        \end{tikzcd}
    \end{equation*}
    commutes, where the left vertical arrow is induced by the inclusion map $G_{L_2} \to G_{L_1}$, and $\N_{L_2 / L_1}$ is the \emph{norm} map.
    We obtain the commutative diagram
    \begin{equation*}
        \begin{tikzcd}
            {G_{L_1}^\ab} & {\widehat{L_1^\times}} \\
            {G_{L_2}^\ab} & {\widehat{L_2^\times}}
            \arrow["\cong"', from=1-2, to=1-1]
            \arrow[from=2-1, to=1-1]
            \arrow["{(\N_{L_2/L_1})^\wedge}"', from=2-2, to=1-2]
            \arrow["\cong"', from=2-2, to=2-1]
        \end{tikzcd}
    \end{equation*}
    by profinite completion, and the isomorphism
    \begin{equation*}
\varprojlim_{H \in \mathscr{H}_m (G_K)} H^\ab
        \xrightarrow{\cong}
            \varprojlim_{L / K} \widehat{L^\times}
    \end{equation*}
    which respects the $G_K$-action, by taking inverse limits.
    Hence
    \begin{equation}
        \label{equation:2025.10.15.11.38.04}
\Gal(K^{m + 1} / K^m)^{G_K}
        \cong
            \varprojlim_{L / K} \widehat{K^\times}
    \end{equation}
    by \cref{lemma:2025.10.15.11.37.39} (3).
    (Note that the limit is taken over the inverse system in which
    the homomorphism
    \begin{equation*}
        (-)^{[L_2 : L_1]}
        \colon \,\,\,
            \widehat{K^\times} \,
            (=(\widehat{L_2^\times})^{G_K}) \,\,\,
        \to \,\,\,
            \widehat{K^\times} \,
            (=(\widehat{L_1^\times})^{G_K})
    \end{equation*}
    is assigned to each pair $L_1 \subseteq L_2$.)

    It remains to show that if $x = \{ x_L \}_{L / K}$ belongs to the right hand side of \eqref{equation:2025.10.15.11.38.04}, then $x_L = 1$ for all $L / K$.
    This can be verified as follows: for all $n \geq 1$, we can always find a finite extension $L_{(n)} / L$ such that $[L_{(n)} : L] = n$ and $L_{(n)} \subseteq K^m$, e.g.,
    \begin{equation*}
        L_{(n)}
        \coloneqq
            L (
                \mu_{
                    {\sharp (\mathfrak{k}_L)}^n - 1
                } (K^\alg)
            ).
    \end{equation*}
    Therefore, $x_L = (x_{L_{(n)}})^n$ for all $n \geq 1$, and hence
    \begin{equation*}
x_L
        \in
            \bigcap_{n \geq 1} (\widehat{K^\times})^n
        =
            \{ 1 \},
    \end{equation*}
    which proves the claim.
\end{proof}

\printbibliography
 
\end{document}